\documentclass{amsart}

\usepackage{amssymb, amsmath}
\usepackage{mathabx}
\usepackage{graphicx}
\usepackage[utf8]{inputenc}
\usepackage[T1]{fontenc}
\usepackage[pdfencoding=auto]{hyperref}
\usepackage{epstopdf}
\usepackage{mathtools}
\usepackage{mathrsfs}

\newtheorem{thm}{Theorem}
\newtheorem*{thm*}{Theorem}

\newtheorem{lem}[thm]{Lemma}
\newtheorem{prop}[thm]{Proposition}

\newtheorem{rem}[thm]{Remark}
\newtheorem{ex}{Example}
\numberwithin{equation}{section}

\newcommand{\Cpx}{\mathbb{C}}

\newcommand{\tq}{\ |\ }
\newcommand{\sma}{\! <\!}
\newcommand{\weyl}{\mathcal{W}}

\newcommand{\ttau}{\boldsymbol{\tau}}
\newcommand{\kk}{\mathbf{k}}
\newcommand{\pp}{\mathbf{p}}
\newcommand{\qq}{\mathbf{q}}
\newcommand{\llr}{\Lambda}

\newcommand{\cnr}{\mathscr{C}}
\newcommand{\extDw}{D^{+}(w)}
\newcommand{\uness}{\mathscr{U}\! }
\newcommand{\shape}[3]{\ensuremath{\mathbf{#1#2#3}}}
\newcommand{\NE}{\mathscr{N}\!\!e}
\newcommand{\op}{\mathcal{O}p_{\ttau}}

\begin{document}

\title{Theta-vexillary signed permutations}

\author{Jordan Lambert}
\address{Department of Mathematics, Federal University of Juiz de Fora, Juiz de Fora 36036-900, Minas Gerais, Brazil}
\email{jordansilva2008@gmail.com}
\thanks{The author was supported by FAPESP Grant 2013/10467-3 and 2014/27042-8}



\begin{abstract}
Theta-vexillary signed permutations are elements in the hyperoctahedral group that index certain classes of degeneracy loci of type B and C. These permutations are described using triples of $s$-tuples of integers subject to specific conditions. The objective of this work is to present different characterizations of theta-vexillary signed permutations, describing them in terms of corners in the Rothe diagram and pattern avoidance.
\end{abstract}

\maketitle

\section{Introduction}

A permutation $w$ is called vexillary if and only if it avoids the patters $[2\ 1\ 4\ 3]$, i.e., there are no indices $a<b<c<d$ such that $w(b)<w(a)<w(d)<w(c)$. Vexillary permutations were found by Lascoux and Sch\"utzenberger \cite{LS} in the 1980s. Fulton \cite{Fu91} in the 1990s obtained other equivalent characterizations for the vexillary permutations: in addition to the pattern avoidance criterion, he figured out one in terms of the essential set of a permutation, among others. Since $S_{n}$ is the Weyl group of type A, the vexillary permutations represent Schubert varieties in some flag manifold where the Lie group is $G=\mathrm{Sl}(n,\Cpx)$.

A few years later, the notion of vexillary permutations in the hyperoctaedral group were introduced by Billy and Lam \cite{BL}. Recently, Anderson and Fulton \cite{AF12, AF14} provided a different characterization for vexillary signed permutations. They defined them through a specific triple of integers: given three $s$-tuple of positive integers $\ttau=(\kk,\pp,\qq)$, where $\kk=(0<k_{1}<\cdots< k_{s})$, $\pp=(p_{1}\geqslant\cdots \geqslant p_{s}>0)$, and $\qq=(q_{1}\geqslant\cdots \geqslant q_{s}>0)$, satisfying $p_{i}-p_{i+1}+q_{i}-q_{i+1}>k_{i+1}-k_{i}$ for $1\leqslant i \leqslant s-1$, one constructs a signed permutation $w=w(\ttau)$. Since the hyperoctahedral group $\weyl_{n}$ can be included in the group $S_{2n+1}$, a signed permutation $w$ in $\weyl_{n}$ is vexillary if and only if its inclusion $\iota(w)$ in $S_{2n+1}$ is a vexillary permutation as mentioned above. Anderson and Fulton in \cite{AF14} characterize vexillary signed permutations in terms of essential sets, pattern avoidance, and Stanley symmetric functions.

In this work, we present a class of signed permutations called \emph{theta-vexillary signed permutations}.
They are defined using a triple of integers $\ttau=(\kk,\pp,\qq)$ where we allow negative values for $\qq$ and satisfy eight different conditions, which will be called a \emph{theta-triple}. The set of theta-vexillary signed permutations is relevant because it contains all vexillary signed permutations and k-Grassmannian permutations, which are the ones associated to the Grassmannian Schubert varieties of type B and C.

Theta-vexillary signed permutations have an important geometric interpretation in terms of degeneracy loci. For our purpose, it is easier to denote the hyperoctahedral group as the Weyl group of type B.
Consider a vector bundle $V$ of rank $2n+1$ over $X$, equipped with a nondegenerate form and two flags of bundles $E_{\bullet}=(E_{p_{1}}\subset E_{p_{2}}\subset \cdots \subset E_{p_{s}}\subset V)$ and $F_{\bullet}=(F_{q_{1}}\subset F_{q_{2}}\subset \cdots \subset F_{q_{s}}\subset V)$ such that: for $q>0$, the subbundles $F_{q}$ are isotropic, of rank $n+1-q$; for $q<0$, $F_{q}$ is coisotropic, of corank $n+q$; and all the subbundles $E_{p}$ are isotropic, of rank $n+1-p$. The degeneracy locus of $\ttau$ is
$$
\Omega_{\ttau}:=\{ x\in X \tq \dim(E_{p_{i}} \cap F_{q_{i}})\geqslant k_{i}, \mbox{ for } 1\leqslant i\leqslant s \}.
$$

Anderson and Fulton in \cite{AF15} figured out that if the triple $\ttau$ is subjet to certain conditions, the cohomology class $[\Omega_{\ttau}]$ is the multi-theta-polynomial $\Theta_{\lambda(\ttau)}$ whose coefficiens are Chern classes of the vector bundles $E_{p_{i}}$ and $F_{q_{i}}$. The polinomial $\Theta_{\lambda(\ttau)}$ derives from the theta-polynomials defined via raising operators by Bush, Kresch, and Tamvakis \cite{BKT} and inspired the name theta-vexillary signed permutations.

The main result of this work provides two other ways to characterize theta-vexillary permutations.
If a permutation $w$ in the Weyl group $\weyl_{n}$ of type B is represented as a matrix of dots in a $(2n+1)\times n$ array of boxes, the (Rothe) extended diagram is the subset of boxes that remains after striking out the boxes weakly south or east of each dot. The southeast (SE) corners in the extended diagram form the set of corners $\cnr(w)$.
One characterization of theta-vexillary signed permutations is the set of corners $\cnr(w)$ is the disjoint union of the set $\NE(w)$ which is composed by all corners that form a piecewise path that goes to the northeast direction, and the set $\uness(w)$ of unessential corners. We also have a characterization via pattern avoidance.

\begin{thm}\label{thm:main}
Let $w$ be a signed permutation. The following are equivalent:
\begin{enumerate}
\item $w$ is theta-vexillary, i.e., there is a triple $\ttau$ such that $w=w(\ttau)$;

\item the set of corners $\cnr(w)$ is the disjoint union
$$
\cnr(w)=\NE(w)\dot{\cup}\uness(w),
$$

\item $w$ avoids the follow thirteen signed patterns 
$[\overline{1}\ 3\ 2]$, 
$[\overline{2}\ 3\ 1]$, 
$[\overline{3}\ 2\ 1]$, 
$[\overline{3}\ 2\ \overline{1}]$, 
$[2\ 1\ 4\ 3]$, 
$[2\ \overline{3}\ 4\ \overline{1}]$, 
$[\overline{2}\ \overline{3}\ 4\ \overline{1}]$, 
$[3\ \overline{4}\ 1\ \overline{2}]$,
$[3\ \overline{4}\ \overline{1}\ \overline{2}]$,
$[\overline{3}\ \overline{4}\ 1\ \overline{2}]$, 
$[\overline{3}\ \overline{4}\ \overline{1}\ \overline{2}]$, 
$[\overline{4}\ 1\ \overline{2}\ 3]$,
and $[\overline{4}\ \overline{1}\ \overline{2}\ 3]$.
\end{enumerate}
\end{thm}

This theorem is consequence of Propositions \ref{prop:svexcore2} and \ref{prop:patavoid} and it is similar to the vexillary signed permutation's version. It is interesting to notice that, comparing to the vexillary case, we admit some SE corners in the diagram that are not in an ordered northeast path, which we call the unessential corners. Besides, the characterization via signed pattern avoidance for the theta-vexillary permutations has eight patterns in common with those for the vexillary case and $[2\ 1]$ is the unique not present in this list.

Considering the pattern avoidance criterion, the set of theta-vexillary signed permutations form a new class of permutations according to the ``Database of Permutation Pattern Avoidance'' maintained by Tenner \cite{patternDB}.

This work is part of my Ph.D. thesis \cite{Lam18}.

\subsubsection*{Acknowledgments}
I would like to express my very great appreciation to David Anderson for his valuable and constructive suggestions during the development of this work while I was a visiting scholar at The Ohio State University. I also thank to Lonardo Rabelo for comments on a previous manuscript.

\section{Signed permutations in \texorpdfstring{$\weyl_{n}$}{Wn}}

The notation present here is the same used in \cite{AF16}. We also refer \cite[\S 8.1]{BB} for further details.

Consider the permutation of positive and negative integers, where the bar over the number denotes the negative sign, and consider the natural order of them
$$
\dots, \overline{n},\dots, \overline{2},\overline{1},0,1,\dots, n,\dots
$$

A \emph{signed permutation} is a permutation $w$ satisfying that $w(\overline{\imath})=\overline{w(i)}$, for each $i$.
A signed permutation belongs to $\weyl_{n}$ if $w(m)=m$ for all $m>n$; this is a group isomorphic to the hyperoctahedral group, the Weyl group of types $B_{n}$ and $C_{n}$. Since $w(\overline{\imath})=\overline{w(i)}$, we just need the positive positions when writing signed permutation in one-line notation, i.e., a permutation $w\in \weyl_{n}$ is represented by $w(1)\ w(2)\ \cdots\ w(n)$. For example, the full form of the signed permutation $w=\overline{2}\ 1\ \overline{3}$ in $\weyl_{3}$ is $3\ \overline{1}\ 2\ 0\ \overline{2}\ 1\ \overline{3}$, but we can omit the values at the position $\overline{3}, \overline{2}, \overline{1}$ and $0$. The group $\weyl_{n}$ is generated by the \emph{simple transpositions} $s_{0},\dots, s_{n}$, where for $i>0$, right-multiplication by $s_{i}$ exchanges entries in positions $i$ and $i+1$, and right-multiplication by $s_{0}$ replaces $w(1)$ with $\overline{w(1)}$. Every signed permutation $w$ can be written as $w=s_{i_{1}}\cdots s_{i_{\ell}}$ such that $\ell$ is minimal; call the number $\ell=\ell(w)$ the \emph{length} of $w$.  This value counts the number of inversions of $w\in\weyl_{n}$, and it is given by the formula
\begin{align}\label{eq:lengthTipoB}
\ell(w) =  \#\{1\leqslant i< j\leqslant n \tq w(i)>w(j)\} + \#\{1\leqslant i\leqslant j\leqslant n \tq w(-i)>w(j)\}.
\end{align}

The element $w_{\circ}^{(n)}= \overline{1}\ \overline{2}\cdots\overline{n}$ is the longest element in $\weyl_{n}$ and it is called the involution of $\weyl_{n}$. Notice that the involution $w_{\circ}^{(n)}$ has length $n^{2}$.

The group of permutations $\weyl_{n}$ can be embedded in the symmetric group $S_{2n+1}$, considering $S_{2n+1}$ the permutations of $\overline{n},\dots, 0,\dots, n$. Indeed, define the \emph{odd} embedding by $\iota:\weyl_{n}\hookrightarrow S_{2n+1}$ where it sends $w=w(1)\ w(2)\ \cdots\ w(n)$ to the permutation
$$
\overline{w(n)}\ \cdots\ \overline{w(2)}\ \overline{w(1)}\ 0\  w(1)\ w(2)\ \cdots\ w(n)
$$
in $S_{2n+1}$. The embedding $\iota$ will be used when it is necessary to highlight that we need the full permutation of $w$.

There is also a \emph{even} embedding $\iota':\weyl_{n}\hookrightarrow S_{2n}$ defined by omitting the value $w(0)=0$.

Considering the natural inclusions $\weyl_{n}\subset \weyl_{n+1}\subset\cdots$, we get the infinite Weyl group $\weyl_{\infty}=\cup \weyl_{n}$. When the value $n$ is understood or irrelevant, we can consider $w$ as an element of $\weyl_{\infty}$. The odd embeddings are compatible with the corresponding inclusions $S_{2n+1}\subset S_{2n+3}\subset\cdots$.

\subsection{Diagram of a permutation in \texorpdfstring{$S_{2n+1}$}{S(2n+1)}}

Let us consider the specific case where the permutation group is $S_{2n+1}$. It is important to consider this case because we need to do some modification in the notation that will be useful for us.
 
Consider a $(2n+1)\times (2n+1)$ arrays of boxes with rows and columns indexed by integers $[\overline{n},n]=\{\overline{n},\dots, \overline{1},0,1,\dots,n\}$ in matrix style. The \emph{permutation matrix} associated to a permutation $w\in S_{2n+1}$ is obtained by placing dots in positions $(w(i),i)$, for all $\overline{n}\leqslant i\leqslant n$, in the array. Again the \emph{diagram} of $w$ is the collection of boxes that remain after removing those which are (weakly) south and east of a dot in the permutation matrix. Observe that the number of boxes in the diagram is equal to the length of the permutation.

The \emph{rank function} of a permutation $w\in S_{2n+1}$ for a pair $(p,q)$, where $\overline{n}\leqslant p, q\leqslant n$, is the number of dots strictly south and weakly west of the box $(q-1,\overline{p})$ in the permutation matrix of $w$. In other words, it will be defined by
\begin{align*}
r_{w}(p,q) &:=\#\{i\leqslant \overline{p} \tq w(i)\geqslant q\},
\end{align*}
for $\overline{n}\leqslant p, q\leqslant n$.

We say that a box $(a,b)$ is a southeast (SE) corner of the diagram of $w$ if $w$ has a descent at $b$, with $a$ lying in the interval of the jump, and $w^{-1}$ has a descent at $a$, with $b$ lying in the interval of the jump. This can be written as
\begin{align}\label{eq:defconerSn}
\begin{aligned}
w(b)> a &\geqslant w(b+1) \quad \mbox{and}\\
w^{-1}(a)> b &\geqslant w^{-1}(a+1).
\end{aligned}
\end{align}

A \emph{corner position} of $w$ is a pair $(p,q)$ such that the box $(q-1,\overline{p})$ is a southeast (SE) corner of the diagram of $w$. The \emph{set of corners} of $w$ is the set $\cnr(w)$ of triples $(k,p,q)$ such that $(p,q)$ is a corner position and $k=r_{w}(p,q)$.

For example, consider $w=\iota(\overline{2}\ 3\ 1)=\overline{1}\ \overline{3}\ 2\ 0\ \overline{2}\ 3\ 1$. Figure \ref{fig:exdiagram} shows the diagram of $w$. The SE corners $(q-1,\overline{p})$ are highlighted and they are filled with the rank function values $r_{w}(p,q)$. In this case, the set of corners is 
$$
\cnr(w)=\{(1,3,\overline{1}), (1,1,2), (3,0,\overline{1}),(2,\overline{2},2)\}.
$$

\begin{figure}[ht]
	\centering
	\includegraphics[scale=0.8]{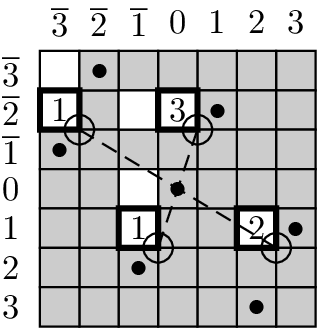}
	\caption[Diagram for $w=\iota(\overline{2}\ 3\ 1)=\overline{1}\ \overline{3}\ 2\ 0\ \overline{2}\ 3\ 1$.]{Diagram for $w=\iota(\overline{2}\ 3\ 1)=\overline{1}\ \overline{3}\ 2\ 0\ \overline{2}\ 3\ 1$. The circle corners connected with dashad lines illustrate the symmetry of Lemma \ref{lema:symmetry}.}
	\label{fig:exdiagram}
\end{figure}

Notice that if a box $(q-1,\overline{p})$ is a SE corner that satisfies \eqref{eq:defconerSn}, then $(p,q)$ is a corner position and $k=r_{w}(p,q)$.

\subsection{Extended diagram of a signed permutation in \texorpdfstring{$\weyl_{n}$}{W(n)}}

We know that signed permutations must satisfy the relation $w(\overline{\imath})=\overline{w(i)}$, then the negative positions can be obtained from the positive ones. Hence, a signed permutation $w\in \weyl_{n}$ corresponds to a $(2n+1)\times n$ array of boxes, with rows indexed by $\{\overline{n},\dots,n\}$ and the columns indexed by $\{\overline{n},\dots,\overline{1}\}$, where the dots are placed in the boxes $(w(i),i)$ for $\overline{n}\leqslant i\leqslant \overline{1}$.

For each dot, we place an ``$\times$'' in those boxes $(a,b)$ such that $a=\overline{w(i)}$ and $i\leqslant b$, in other words, an $\times$ is placed in the same column and opposite along with the boxes to the right of this $\times$.

The \emph{extended diagram} $\extDw$ of a signed permutation $w$ is the collection of boxes in the $(2n+1)\times n$ rectangle that remain after removing those which are south or east of a dot. The \emph{diagram} $D(w)\subseteq \extDw$ is obtained from extended diagram $\extDw$ by removing the ones marked with $\times$. Namely, $D(w)$ is defined by
\begin{align*}
D(w)=\{(i,j)\in [\overline{n},\overline 1]\times[\overline{n},n] \tq w(i)>j, w^{-1}(j)>i \mbox{, and } w^{-1}(-j)>i\}.
\end{align*}

\begin{lem}
The number of boxes of $D(w)$ is equal to the length of $w$.
\end{lem}
\begin{proof}
Observe that if we define  $J = \{ j\in [\overline{n},n] \tq w^{-1}(j)<0\}$, the set $D(w)$ can be split into two subsets $D_{1}(w) = \{(i,j) \in D(w) \tq j\in J\}$ and  $D_{2}(w) = \{(i,j) \in D(w) \tq j\not\in J\}$. Since both sets have cardinality $\# D_{1}(w) = \#\{1\leqslant l<m\leqslant n \tq w(l)>w(m)\}$ and $\# D_{2}(w) = \#\{1\leqslant l\leqslant m\leqslant n \tq w(-l)>w(m)\}$, the assertion follows from Equation \eqref{eq:lengthTipoB}.
\end{proof}

Observe that if we use the embedding $\iota:\weyl_{n}\hookrightarrow S_{2n+1}$, the matrix and extended diagram of $w\in \weyl_{n}$ corresponds, respectively, to the first $n$ columns of the matrix and diagram of $\iota(w)$. The notation $\iota(\extDw)$ will be used when we need to use the respective $(2n+1)\times (2n+1)$ diagram of $\iota(w)$.

The rank function of a permutation $w$ in $\weyl_{n}$ is defined by
\begin{align*}
r_{w}(p,q) &=\#\{i\geqslant p \tq w(i)\leqslant \overline{q}\},
\end{align*}
for $1\leqslant p\leqslant n$, and $\overline{n}\leqslant q\leqslant n$.

Since $w(\overline{\imath})=\overline{w(i)}$, then the rank function $r_{w}(p,q)$ is also equal to $\#\{i\leqslant \overline{p} \tq w(i)\geqslant q\}$, so the rank functions $r_{w}$ coincides to $r_{\iota(w)}$.

Given $w\in \weyl_{n}$, the next lemma states that there is a symmetry about the origin of the corner positions corresponding to $\iota(w)$. 
In order to simplify the notation, given a triple $(k,p,q)$, define the \emph{reflected} triple $(k,p,q)^{\perp}=(k+p+q-1,\overline{p}+1,\overline{q}+1)$.

\begin{lem}[\cite{AF16}, Lemma 1.1]\label{lema:symmetry}
For $w\in \weyl_{n}$, the set of corners  of $\iota(w)\in S_{2n+1}$ has the following symmetry: $(k,p,q)$ is in $\cnr(\iota(w))$ if and only if $(k,p,q)^{\perp}$ is in $\cnr(\iota(w))$.
\end{lem}

We can see in Figure \ref{fig:exdiagram} that both corners in the left half of the diagram are symmetric by the origin to other two corners in the right side.
This behavior will happen for every signed permutation $w$, implying that half of $\cnr(\iota(w))$ suffices to determine the signed permutation $w$; we will consider those corners appearing in the first $n$ columns.

A \emph{corner position} of signed permutation $w$ is a pair $(p,q)$ such that the box $(q-1,\overline{p})$ is a southeast (SE) corner of the extended diagram of $w$. The \emph{set of corners} $\cnr(w)$ of a signed permutation $w$ is the set of triples $(k,p,q)$ such that $(q-1,\overline{p})$ is a SE corner of the extended diagram $\extDw$ and $k=r_{w}(p,q)$, except for corner positions $(p,q)$ where $p=1$ and $q<0$. This exception comes from the fact that $(1,q)$, for $q<0$, is not a corner position in $\iota(w)$ because the respective box $(q-1,\overline{1})$ cannot be a SE corner since $w(0)=0$. 

\begin{rem}
Anderson and Fulton in \cite{AF16, AF14} defined a slightly different set called \emph{essential set} of $w$. This set is contained in the set of corners, since they remove a few ``redundant''  SE corners. In the present work, we need the whole set of corners since the essential set is not enough to perform our computations.
\end{rem}

Since the integer $k$ is the rank of $w$ in $(p,q)$, sometimes we can simply say that the corner position $(p,q)\in\cnr(w)$, instead of the triple $(k,p,q)$.

The Figure \ref{fig:diagram1} illustrates the extended diagram and the set of cornet of the signed permutation $w=10\ 1\ 5\ 3\ \overline{2}\ \overline{4}\ 6\ \overline{9}\ \overline{8}\ \overline{7}$.

\begin{figure}[ht]
	\centering
	\includegraphics[scale=0.7]{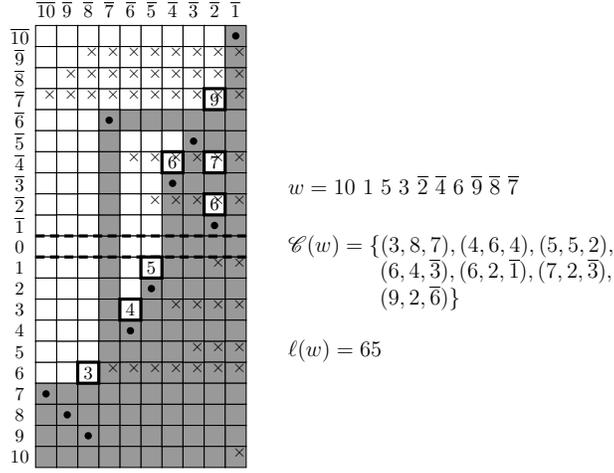}
	\caption{Diagram and set of corners of a signed permutation.}
	\label{fig:diagram1}
\end{figure}
 
To make the diagrams look cleaner, from now on we won't denote $\times$ in the extended diagrams $\extDw$. During the text, it can happen that we omit the word ``extended'' since we are only interested in studying the extended diagram of a signed permutation $w$ so that the diagram $D(w)$ won't be useful for us. 

\subsection{NE path and unessential corners}

Suppose that $w\in \weyl_{n}$ is any signed permutation. There are two notable classes of SE corner in the set $\cnr(w)$ that we will be important to our main theorem. They are the corners in the northeast path and the unessential corners.

Given any signed permutation $w$, consider a \emph{(strict) partial order} for the set of corners $\cnr(w)$ by $(p,q)<(p',q')$ if and only if $p>p'$ and $q<q'$, for corner positions $(p,q),(p',q')\in \cnr(w)$. For example, in Figure \ref{fig:diagram1}, the unique possible relation is $(4,\overline{3})<(2,\overline{2})$, the two boxes filled in with the value $6$. 

Define the \emph{northeast (NE) path} as the set $\NE(w)$ of minimal elements of $\cnr(w)$ relative to the poset ``$<$''. Using the same example of Figure \ref{fig:diagram1}, we have that $\NE(w)=\cnr(w)-\{(6,2,\overline{1})\}$, since all the corners are minimal under this poset except $(6,2,\overline{1})$.

The positions $(p_{i},q_{i})$ of the NE path $\NE(w)$ can be ordered so that $p_{1}\geqslant p_{2}\geqslant \cdots \geqslant p_{r}>0$ and $q_{1}\geqslant q_{2}\geqslant \cdots \geqslant q_{r}$. In fact, suppose that we order $p_{1}\geqslant p_{2}\geqslant \cdots \geqslant p_{r}>0$ but there is $i$ such that and $q_{i}<q_{i+1}$. If $p_{i}=p_{i+1}$ then we can exchange $i$ and $i+1$. Otherwise, if $p_{i}>p_{i+1}$ then $(p_{i},q_{i})<(p_{i+1},q_{i+1})$ and $(p_{i+1},q_{i+1})$ does not belong to the NE path.

Given a signed permutation $w$, we say that a corner position $(p,q)$ of $\cnr(w)$ is \emph{unessential} if there are corner positions $(p_{1},q_{1})$, $(p_{2},q_{2})$ and $(p_{3},q_{3})$ in the NE path $\NE(w)$ satisfying the following conditions:
\begin{gather*}
p_{1}=p \mbox{ and } q_{1}<q<0;\\
p_{2}>0 \mbox{ and } q_{2}=\overline{q}+1;\\
(p_{3},q_{3})<(p,q).
\end{gather*}

In other words, $(p,q)$ is not a minimal corner in the poset in the upper half of the diagram, the box $(q_{1}-1,\overline{p_{1}})$ lies above and in the same column of the box $(q-1,\overline{p})$, and the box $(\overline{q_{2}},p_{2}-1)$ reflected from $(q_{2}-1,\overline{p_{2}})$ lies to the right and in the same row of $(q-1,\overline{p})$, as shown in Figure \ref{fig:configuness}.
We denote by $\uness(w)$ the set of all unessential corners of $w$.

\begin{figure}[ht]
	\centering
	\includegraphics[scale=0.8]{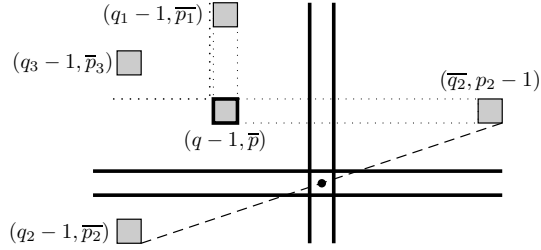}
	\caption[Configuration of an unessential corner $(p,q)$.]{Configuration of an unessential corner $(p,q)$. The highlighted box $(q-1,\overline{p})$ satisfies all the conditions.}
	\label{fig:configuness}
\end{figure}

It is important to emphasize that all three corners $(p_{1},q_{1})$, $(p_{2},q_{2})$ and $(p_{3},q_{3})$ must belong to the NE path $\NE(w)$.

For instance, considering the signed permutation $w=10\ 1\ 5\allowbreak\ 3\ \overline{2}\ \overline{4}\ 6\ \overline{9}\ \overline{8}\ \overline{7}$ of Figure \ref{fig:diagram1}, the set of unessential corners $\uness(w)$ only contains the triple $(6,2,\overline{1})$.

\section{Theta-triples and theta-vexillary signed  permutations}\label{sec:svex}


A \emph{theta-triple} is three $s$-tuples $\ttau=(\kk,\pp,\qq)$ with
\begin{align}\label{eq:defvex}
\kk & =(0< k_{1}<k_{2}<\dots < k_{s}), \nonumber\\
\pp& =(p_{1} \geqslant p_{2} \geqslant \dots \geqslant p_{s} >0),\\
\qq & =(q_{1} \geqslant q_{2} \geqslant \dots \geqslant q_{s}), \nonumber
\end{align}
and satisfying eight conditions. We intentionally split such conditions in three blocks that share common characteristics. The first three conditions are
\begin{enumerate}
\item[A1.] $q_{i}\neq 0$ for all $i$;
\item[A2.] $q_{i}\neq-q_{j}$, for any $i\neq j$.
\item[A3.] If $q_{s}<0$ then $p_{s}>1$;
\end{enumerate}

Now, let $a=a(\ttau)$ be the integer such that $q_{a-1}>0>q_{a}$, allowing $a=1$ and $a=s+1$ for the cases where all $q$'s are negative or all $q$'s are positive, respectively. For all $i\geqslant a$, denote by $R(i)$ (or $R(i)_{\ttau}$ to specify the triple) the unique integer such that $q_{R(i)}>-q_{i}>q_{R(i)+1}$; if necessary, consider $k_{0}=0$, $p_{0}=+\infty$, $q_{0}=+\infty$, and $R(a-1)=a-1$.
The next three conditions are
\begin{enumerate}
\item[B1.] $(p_{i}-p_{i+1})+(q_{i}-q_{i+1})> k_{i+1}-k_{i}$, for $1\leqslant i<a-1$;
\item[B2.] $(p_{i}-p_{i+1})+(q_{i}-q_{i+1})> (k_{i+1}-k_{i})+(k_{R(i)}-k_{R(i+1)})$, for $a\leqslant i<s$;
\item[B3.] $p_{s}+q_{s}+k_{s}> k_{R(s)}+1$, if $a \leqslant s$.
\end{enumerate}

It is important to observe that none of the above conditions compare indexes $a-1$ and $a$. Finally, consider $a\leqslant i\leqslant s$ and let $L(i)=L_{\ttau}(i)$ be the biggest integer $j$ in $\{R(i)+1,\dots, a-1\}$ satisfying $k_{j}-k_{R(i)+1}\geqslant q_{R(i)+1}-q_{j}$, i.e.,
\begin{align}
L(i)=\max\{R(i)+1\leqslant j \leqslant a-1 \tq k_{j}-k_{R(i)+1}\geqslant q_{R(i)+1}-q_{j}\}.
\end{align}

The last two conditions are
\begin{enumerate}
\item[C1.] $-q_{i}\geqslant k_{i}-k_{R(i)}$ for all $a\leqslant i\leqslant s$;
\item[C2.] $-q_{i}\geqslant {q_{L(i)}+k_{L(i)}-k_{R(i)}}$ for all $a\leqslant i\leqslant s$.
\end{enumerate}

Given a theta-triple $\ttau$, \emph{the construction algorithm of the permutation $w(\ttau)$} is given by a sequence of $s+1$ steps as follows:
\begin{description}
\item [Step ($1$)] Starting in the $p_{1}$ position, place $k_{1}$ consecutive entries, in increasing order, ending with $-q_{1}$. Mark the \emph{absolute} value of these numbers as ``used'';

\item [Step ($i$)] For $1<i\leqslant s$, starting in the $p_{i}$ position, or the next available position to the right, fill the next available $k_{i}-k_{i-1}$ positions with entries chosen consecutively from the unused \emph{absolute} numbers, in increasing order, ending with $-q_{i}$ or, if it is not available, the biggest unused number below $-q_{i}$. Again, mark the absolute value of these numbers as ``used'';

\item [Step ($s+1$)] Fill the remaining available positions with the unused positive numbers, in increasing order.
\end{description}

Notice that we should mark as used the absolute of the placed values because we allow negative $q_{i}$ for a theta-triple.

A signed permutation $w\in \weyl_{n}$ is called \emph{theta-vexillary} if $w=w(\ttau)$ comes from some theta-triple $\ttau=(\kk,\pp,\qq)$.

\begin{ex}\label{ex:perm1}
The permutation $w$ given in Figure \ref{fig:diagram1} can be obtained from the triple $\ttau=(3\: 4\: 5\: 6\: 9,\: 8\: 6\: 5\: 4\: 2,\: 7\: 4\: 2\: \overline{3}\: \overline{6})$ using the above algorithm:
$$
\begin{array}{cccccccccccc}
 & \cdot & \cdot & \cdot & \cdot & \cdot & \cdot & \cdot & \mathbf{\overline{9}} & \mathbf{\overline{8}} & \mathbf{\overline{7}} \\ 
 & \cdot & \cdot & \cdot & \cdot & \cdot & \mathbf{\overline{4}} & \cdot & {\overline{9}} & {\overline{8}} & {\overline{7}} \\ 
 & \cdot & \cdot & \cdot & \cdot & \mathbf{\overline{2}} & {\overline{4}} & \cdot & {\overline{9}} & {\overline{8}} & {\overline{7}} \\  
 & \cdot & \cdot & \cdot & \mathbf{3} & {\overline{2}} & {\overline{4}} & \cdot & {\overline{9}} & {\overline{8}} & {\overline{7}} \\  
 & \cdot & \mathbf{1} & \mathbf{5} & {3} & {\overline{2}} & {\overline{4}} & \mathbf{6} & {\overline{9}} & {\overline{8}} & {\overline{7}} \\  
w=  & \mathbf{10} & {1} & {5} & {3} & {\overline{2}} & {\overline{4}} & {6} & {\overline{9}} & {\overline{8}} & {\overline{7}}
 \end{array}
$$

Clearly, $\ttau$ satisfies all eight conditions above and, then, $\ttau$ is a theta-triple.
Thus, $w$ is theta-vexillary signed permutation. Observe that every pair $(p_{i},q_{i})$ in this triple is also a corner position in the diagram of Figure \ref{fig:diagram1}. This fact is not a coincidence, and we will show that every theta-triple are corner positions on the permutation.
\end{ex}

Notice the construction algorithm does not create an inversion inside a step, i.e., if $a<b$ are positions placed by a Step $(i)$ then $w(a)<w(b)$.

\begin{rem}
The definition of a theta-triple was motivated by the \emph{triple of type C} given by Anderson and Fulton \cite{AF15}. Indeed, any theta-triple is a triple of type C, but the converse is not true. A theta-triple has two properties that a triple of type C does not: each $(k_{i},p_{i},q_{i})$ is associated to a SE corner of $w(\ttau)$ (Proposition \ref{prop:SEcorners}); and any theta-vexillary permutation is given by a unique theta-triple $\ttau$ (Proposition \ref{prop:unique}).
Both results are relevant  when we study SE corners in the diagram of $w(\ttau)$.
\end{rem}

Now, we will give a brief explanation about the eight conditions of a theta-triple.
Conditions A1, A2 and A3  guarantee that the permutation $w(\ttau)$ associated to such theta-triple is a signed permutation.

Conditions B1, B2 and B3, in some sense, characterize a theta-vexillary permutations as well as the condition $(p_{i}-p_{i+1})+(q_{i}-q_{i+1})> k_{i+1}-k_{i}$ does for vexillary permutations. In condition B2, an extra $k_{R(i)}-k_{R(i+1)}$ is added to the right side because, during the construction of $w(\ttau)$, Step $(i)$ skips an equal number of entries that have already been used in previous steps. Moreover, condition B3 is equivalent to apply $i=s$ in condition B2, where we consider the extreme cases $(k_{0},p_{0},q_{0})=(0,n,n)$ and $(k_{s+1},p_{s+1},q_{s+1})=(n,1,-n)$.

Finally, for conditions C1 and C2, the next lemma states an equivalent definition of them based on the construction algorithm of $w(\ttau)$:

\begin{lem}\label{lema:posEntriesStep} 
The conditions \emph{C1} and \emph{C2} are equivalent, respectively, to
\begin{enumerate}
\item[{C1$'$.}] Given any $a\leqslant i \leqslant s$, all entries placed by Steps $(a)$ to $(i)$ are positive;

\item[{C2$'$.}] Given any $a\leqslant i \leqslant s$, all entries placed by Steps $(R(i)+1)$ to $(a-1)$ are strictly bigger than $q_{i}$.
\end{enumerate}
\end{lem}
\begin{proof}
For the first statement, observe that all Steps from $(a)$ to $(i)$ must skip at most $k_{a-1}-k_{R(i)}$ values because they were already used in Steps $(R(i)+1)$ to $(a-1)$ and denote by $\alpha:=-q_{i}-(k_{a-1}-k_{R(i)})$ the number of available positive entries from $1$ to $\overline{q_{i}}$ that can be used by Steps $(a)$ to $(i)$. Then, condition C1 is equivalent to say that $\alpha\geqslant k_{i}-k_{a-1}$, i.e., there is enough positive values available to be placed by Steps $(a)$ to $(i)$. 

For the second assertion, remember that the definition of $L(i)$ says that it is the biggest integer in $\{R(i)+1,\dots, a-1\}$ where $k_{L(i)}-k_{R(i)+1}\geqslant q_{R(i)+1}-q_{L(i)}$. The smallest possible entry placed by Steps $(R(i)+1)$ to $(L(i))$ is limited below by $\overline{q_{L(i)}+k_{L(i)}-k_{R(i)}}+1$. Since for any Step $(j)$ after $L(i)$, we have that $k_{j}-k_{R(i)+1}<q_{R(i)+1}-q_{j}$, then no entry placed by such step cannot be smaller than $\overline{q_{L(i)}}$. So, every entry placed by Steps $(R(i)+1)$ to $(a-1)$ is limited below by $\overline{q_{L(i)}+k_{L(i)}-k_{R(i)}}+1$, and we conclude that both conditions C2 and C2$'$ imply that $q_{i}< \overline{q_{L(i)}+k_{L(i)}-k_{R(i)}}+1$.
\end{proof}

In other words, conditions C1 and C2 guarantee that given $i\geqslant a$, then all values placed by Steps $(R(i)+1)$ to $(i)$ ranges from ${q_{i}}$ to $\overline{q_{i}}$. In practice, it will be easier to use C1$'$ and C2$'$ instead of C1 and C2.

Now, let us study the descents of a theta-vexillary signed permutation $w(\ttau)$ and its inverse $w(\ttau)^{-1}$. 

\begin{prop}\label{prop:descentsofW}
Let $w=w(\ttau)$ be a theta-vexillary signed permutation and $\ttau=(\kk,\pp,\qq)$ be a theta-triple. Then all the descents of $w$ are at positions $p_{i}-1$, i.e, for each $i$, we have $w(p_{i}-1)>\overline{q_{i}}\geqslant w(p_{i})$ and there are no other descents.
\end{prop} 
\begin{proof}
In Step $(1)$, no descents are created, unless $p_{1}=1$, in which case the permutation has a single descent at 0. For $1<i<a$, this is proved in Lemma 2.2 of \cite{AF14}.
Now, supposing that $a\leqslant i\leqslant s$ and $i\geqslant 2$, assume inductively that for $j<i$, there is a descent at position $p_{j}-1$ whenever this positions has been filled, satisfying $w(p_{j}-1)>\overline{q_{j}}\geqslant w(p_{j})$, and there are no other descents. By Lemma \ref{lema:posEntriesStep}, only positive entries are placed in consecutive vacant positions of Step $(i)$, from left to right, at position $p_{i}$ (or the next vacant position to the right, if $p_{i-1}=p_{i}$). We consider ``sub-steps'' of Step $(i)$, where we are placing an entry at position $p\geqslant p_{i}$, and distinguish three cases.
First, suppose we are at position $p$, with $p<p_{i-1}-1$. In this case, the previous entry placed in Step $(i)$ (if any) was placed at position $p-1$, so we did not create a descent at $p-1$. Position $p+1$ is still vacant, so no new descents are created.

To clarify this proof, let $\ttau=(3\: 4\: 5\: 6\: 9,\: 8\: 6\: 5\: 4\: 2,\: 7\: 4\: 2\: \overline{3}\: \overline{6})$ as in Example \ref{ex:perm1}. In Step $(5)$, the first entry placed is $1$ and it does not create a descent:
$$
w= \cdot \: \mathbf{1}\: \cdot \: {3}\: {\overline{2}} \: {\overline{4}} \: \cdot \: {\overline{9}} \: {\overline{8}} \: {\overline{7}} 
$$

Next, suppose we are at position $p=p_{i-1}-1$. This means that $p_{i-1}-p_{i}\leqslant k_{i}-k_{i-1}$, so let $\beta=(k_{i}-k_{i-1})-(p_{i-1}-p_{i})$ be the number of entries remaining to be placed in Step $(i)$, after placing the current at position $p$. 
Condition B1 tell us that $q_{i}\leqslant q_{i}+\beta<q_{i-1}$, then considering the integer interval $\mathcal{I}_{i}=\{\overline{q_{i-1}}+1,\dots,\overline{q_{i}}\}$, it must be non-empty. We claim that the entry $w(p)=w(p_{i-1}-1)$  lies in $\mathcal{I}_{i}$ and therefore $w(p_{i-1}-1)>\overline{q_{i-1}}\geqslant w(p_{i-1})$, proving this situation. Remember that the construction algorithm must skip those entries that its absolute value have already been used, and then this claim is equivalent to say that even removing from $\mathcal{I}_{i}$ those repetitions, there still is some value to be picked by $w(p)$ in $\mathcal{I}_{i}$.

To prove that claim, lets count how many values in $\mathcal I_{i}$ were used in previous steps. For $a\leqslant j<i$, any entry $x$ of Steps $(j)$ satisfies $x\leqslant \overline{q_{j}}\leqslant \overline{q_{i-1}}$, that means $x\not\in \mathcal{I}_{i}$. If $1\leqslant j\leqslant R(i)$ then any entry $x$ placed in Steps $(j)$ satisfies $x\leqslant\overline{q_{j}}\leqslant\overline{q_{R(i)}}<q_{i}$, implying that $\overline{x}\not\in \mathcal{I}_{i}$. If $R(i-1)<j<a$ then by condition C2$'$, any entry $x$ placed in Steps $(j)$ satisfies $x>q_{i-1}$, implying that $\overline{x}\not\in \mathcal{I}_{i}$. Finally, if $R(i)<j\leqslant R(i-1)$ then by condition C2$'$, any entry $x$ placed in Step $(j)$ satisfies $q_{i}<x\leqslant\overline{q_{j}}\leqslant\overline{q_{R(i-1)}}<q_{i-1}$, hence, $\overline{x}\in \mathcal{I}_{i}$. 
We conclude that the only absolute values placed in previous steps that belongs to the interval $\mathcal{I}_{i}$ are all the ones from Steps $(R(i)+1)$ to $(R(i-1))$. So there are $\alpha:=k_{R(i-1)}-k_{R(i)}$ values in $\mathcal{I}_{i}$ that cannot be used in Step $(i)$ in position $p$.
In order to place the correct value for position $p$ of Step $(i)$, we need to consider that the values which are going to be placed after position $p$ also must belong to $\mathcal{I}_{i}$ and are bigger than $w(p)$, i.e., it also is required to skip the $\beta$ biggest values in $\mathcal{I}_{i}$. Since the number of elements of $\mathcal{I}_{i}$ is $\overline{q_{i}}-\overline{q_{i-1}}$, follows from condition B2 that $\#(\mathcal{I}_{i})>\alpha+\beta$ and, therefore, there is some value in $\mathcal{I}_{i}$ to pick for $w(p)$.

Continuing the example, in Step $(5)$, the second entry placed is $5$, which creates a descent at position $3$:
$$
w= \cdot \: {1}\: \mathbf{5} \: {3}\: {\overline{2}} \: {\overline{4}} \: \cdot \: {\overline{9}} \: {\overline{8}} \: {\overline{7}} 
$$

Finally, suppose we are at position $p\geqslant p_{i-1}$. Using the previous case, the entry to be placed is some $x\in \mathcal{I}_{i}$. When an entry is placed in a vacant position to the right of a filled position, it does not create a descent since either all entries already placed the previous steps are smaller than $\overline{q_{i-1}}< x$ or the entries placed in this step is smaller than $x$. When it is placed to the left of a filled position, which can only happen at positions $p_{j}-1$ for some $j<i-1$, and it does create a descent at the position $p_{j}-1$ satisfying $w(p_{j}-1)>\overline{q_{i-1}}\geqslant \overline{q_{j}}\geqslant w(p_{j})$

In Step $(5)$ of our example, it remains to place the 3rd value $6$ in the next vacant position, which occurs at position $7$. Observe that we do not create a descent at  the filled position to its left, but we do create a descent at position $7$, since the position $8$ is already filled:
$$
w= \cdot \: {1}\: {5} \: {3}\: {\overline{2}} \: {\overline{4}} \: \mathbf{6} \: {\overline{9}} \: {\overline{8}} \: {\overline{7}} 
$$

At Step $(s+1)$, we can apply the previous case for $i=s+1$, adding the values $k_{s+1}=n$, $p_{s+1}=0, q_{s+1}=-n+1$ to $\ttau$. This procedure will create descents only at those $p_{j}-1$ which are still vacant.
\end{proof}

Given a triple $\ttau=(\kk,\pp,\qq)$, the dual triple is defined by $\ttau^{*}=(\kk,\qq,\pp)$, where $p$ and $q$ were switched. Clearly, a dual triple could not be a theta-vexillary permutation, but the dual triple is useful to compute the inverse of $w(\ttau)$.

The dual triple $\ttau^{*}=(k,q,p)$ determines a signed permutation $\iota(w(\ttau^{*}))$ in $ S_{2n+1}$ using the following algorithm:
\begin{description}
\item [Step (0)] Put a zero at the position $0$;
\item [Step (1)] Starting in the $q_{1}$ position, place $k_{1}$ consecutive entries, in increasing order, ending with $-p_{1}$. Mark the absolute value of these numbers as ``used'' and fill the reflection through $0$ with the respective reflection $w(\overline{a})=\overline{w(a)}$;
\item [Step (i)] For $1<i\leqslant s$, starting in the $q_{i}$ position (if $q_{i}<0$ then use a position before zero), or the next available position to the right, fill the next available $k_{i}-k_{i-1}$ positions with entries chosen consecutively from the unused absolute numbers, in increasing order, ending with $-p_{i}$ or, if it is not available, the biggest unused number below $-p_{i}$. Again, mark the absolute value of these numbers as ``used'' and fill the reflection through $0$ with the respective reflection $w(\overline{a})=\overline{w(a)}$;

\item [Step (s+1)] Fill the remaining available positions after $0$ with the unused positive numbers in increasing order. Finally, fill the reflection through $0$ with the respective reflection $w(\overline{a})=\overline{w(a)}$
\end{description}

The difference here compared to the construction using the theta-vexillary permutation is that we allow to have negative positions, so we need the full form of the permutation. The signed permutation $w(\ttau^{*})$ is obtained from $\iota(w(\ttau^{*}))$ by restricting it to the positions $\{1,\dots,n\}$.

\begin{lem}
We have $w(\ttau^{*})=w(\ttau)^{-1}$.
\end{lem}
\begin{proof}
We can prove in the same way as Lemma 2.3 of \cite{AF14}, adding the fact that for $a\leqslant i\leqslant s$, the permutation $\iota(w)$ maps the set $a(i)$ to $b(i)$ and, hence, the inverse $\iota(w)^{-1}$ maps $b(i)$ to $a(i)$.
\end{proof}

\begin{ex}
Consider the dual triple $\ttau^{*}=(3\: 4\: 5\: 6\: 9,\: 7\: 4\: 2\: \overline{3}\: \overline{6},\: 8\: 6\: 5\: 4\: 2)$ of the one gave in Example \ref{ex:perm1}. The permutation $\iota(w(\ttau^{*}))$  is constructed as follows:
$$
\begin{array}{ccccccccccc|c|cccccccccc}
  & \cdot & \cdot & \cdot & \cdot & \cdot & \cdot & \cdot & \cdot & \cdot & \cdot & \mathbf{0} & \cdot & \cdot & \cdot  & \cdot & \cdot & \cdot & \cdot & \cdot & \cdot & \cdot\\
 
 & \cdot & \mathit{8} & \mathit{9} & \mathit{10} & \cdot & \cdot & \cdot & \cdot & \cdot & \cdot & {0} & \cdot & \cdot & \cdot  & \cdot & \cdot & \cdot & \mathbf{\overline{10}} & \mathbf{\overline{9}} & \mathbf{\overline{8}} & \cdot\\

 & \cdot & {8} & {9} & {10} & \cdot & \cdot & \mathit{6} & \cdot & \cdot & \cdot & {0} & \cdot & \cdot & \cdot  & \mathbf{\overline{6}} & \cdot & \cdot & {\overline{10}} & {\overline{9}} & {\overline{8}} & \cdot\\

 & \cdot & {8} & {9} & {10} & \cdot & \cdot & {6} & \cdot & \mathit{5} & \cdot & {0} & \cdot & \mathbf{\overline{5}} & \cdot  & {\overline{6}} & \cdot & \cdot & {\overline{10}} & {\overline{9}} & {\overline{8}} & \cdot\\

 & \cdot & {8} & {9} & {10} & \cdot & \cdot & {6} & \mathbf{\overline{4}} & {5} & \cdot & {0} & \cdot & {\overline{5}} & \mathit{4}  & {\overline{6}} & \cdot & \cdot & {\overline{10}} & {\overline{9}} & {\overline{8}} & \cdot\\

 & \cdot & {8} & {9} & {10} & \mathbf{\overline{7}} & \mathbf{\overline{3}} & {6} & {\overline{4}} & {5} & \mathbf{\overline{2}} & {0} & \mathit{2} & {\overline{5}} & {4}  & {\overline{6}} & \mathit{3} & \mathit{7} & {\overline{10}} & {\overline{9}} & {\overline{8}} & \cdot\\

 & \mathit{\overline{1}} & {8} & {9} & {10} & {\overline{7}} & {\overline{3}} & {6} & {\overline{4}} & {5} & {\overline{2}} & {0} & {2} & {\overline{5}} & {4}  & {\overline{6}} & {3} & {7} & {\overline{10}} & {\overline{9}} & {\overline{8}} & \mathbf{1}\\
 \end{array}
$$

For each step, the bold numbers represent the values placed for such step, and the italic ones are their reflection through zero.
Thus, $w(\ttau^{*})= 2 \: \overline{5} \: 4 \: \overline{6} \: 3 \: 7 \: \overline{10} \: \overline{9} \: \overline{8} \: 1$ and we can easily verify that this permutation is the inverse  of
$w=10\ 1\ 5\ 3\ \overline{2}\ \overline{4}\ 6\ \overline{9}\ \overline{8}\ \overline{7}$.
\end{ex}

Although $w(\ttau^{*})$ is not theta-vexillary, a similar version of Proposition \ref{prop:descentsofW} holds for this case and the proof follows that same idea.

\begin{prop}\label{prop:descentsofWinv}
Let $w=w(\ttau)$ be a theta-vexillary signed permutation, for a theta-triple $\ttau=(\kk,\pp,\qq)$. Then all the descents of $w^{-1}$ are at positions $q_{i}-1$, when $i<a$, and $\overline{q_{i}}$, when $i\geqslant a$. In fact, we have
\begin{align*}
&w^{-1}(q_{i}-1)>\overline{p_{i}}\geqslant w^{-1}(q_{i}) \mbox{ , for } i<a;\\
&w^{-1}(\overline{q_{i}})>p_{i}-1\geqslant w^{-1}(\overline{q_{i}}+1) \mbox{ , for } i\geqslant a;
\end{align*}
and there are no other descents.
\end{prop}

\section{Extended diagrams for theta-vexillary permutations}

In this section, we aim to understand how a theta-vexillary permutation looks like in the extended diagram.

Given a position $(p,q)$ in the extended diagram $D^{+}(w)$, define the \emph{left lower region} of $(p,q)$ by the set boxes in the extended diagram strictly south and weakly west of the SE corner $(q-1,\overline{p})$. In other words, denoting it by $\llr(p,q)$, this set is 
\begin{align*}
\llr(p,q):=\{(a,b) \in D^{+}(w) \tq a\geqslant {q}, b\leqslant \overline{p}\}. 
\end{align*}

Notice that the construction algorithm of a theta-vexillary permutation $w(\ttau)$ can also be seen as a process of placing dots in the extended diagram, since each pair $(w(i),i)$ corresponds to a dot in the diagram. We can say that a Step $(i)$ places dots in the diagram using the following rule: if an entry $x$ is placed at a position $z$ in the permutation, i.e., $w(z)=x$, then it produces a dot at the box $(\overline{x},\overline{z})$ in the diagram. For instance, the triple $\ttau=(3\: 4\: 5\: 6\: 9,\: 8\: 6\: 5\: 4\: 2,\: 7\: 4\: 2\: \overline{3}\: \overline{6})$ of Example \ref{ex:perm1} whose diagram is represented in Figure \ref{fig:diagram1}. The first step places the entries $\overline{9}$, $\overline{8}$ and $\overline{7}$, respectively, at positions $8$, $9$ and $10$, which correspond to place dots in boxes $(9,\overline{8})$, $(8,\overline{9})$ and $(7,\overline{10})$. The second step places only a dot in the box $(4,\overline{6})$. The next steps places all other dots in the diagram.

\begin{prop}\label{prop:SEcorners}
Let $w=w(\ttau)$ be a theta-vexillary signed permutation and $\ttau=(\kk,\pp,\qq)$ be a theta-triple. Then we have the following:
\begin{enumerate}

\item The boxes $(q_{i}-1,\overline{p_{i}})$ and their reflection $(\overline{q_{i}}, p_{i}-1)$ are SE corners of the diagram of $\iota(w)$ (not necessary all of them);

\item For any $1\leqslant i\leqslant s+1$, all the dots placed by Step $(i)$ in the diagram are inside region $\llr(p_{i},q_{i})$ and outside $\llr(p_{i-1},q_{i-1})$;

\item $k_{i}$ is the number of dots inside the region $\llr(p_{i},q_{i})$.
\end{enumerate}
\end{prop}
\begin{proof}

Lemma \ref{lema:symmetry} says that there is a symmetry between boxes $(q_{i}-1,\overline{p_{i}})$ and their reflection $(\overline{q_{i}}, p_{i}-1)$. Then, it suffices to prove that every $(q_{i}-1,\overline{p_{i}})$ is a SE corner. If $p>0$, then a signed permutation $w$ has a descent at position $p-1$ if and only if $i(w)$ has descents at position $p-1$ and $\overline{p}$. By proposition \ref{prop:descentsofW}, $\iota(w)$ satisfies $\iota(w)(p_{i}-1)>\overline{q_{i}}\geqslant \iota(w)(p_{i})$, and it implies that
\begin{align*}
\iota(w)(\overline{p_{i}})>{q_{i}} -1\geqslant \iota(w)(\overline{p_{i}}+1).
\end{align*}

On the other hand, by Proposition \ref{prop:descentsofWinv}, $\iota(w)^{-1}$ satisfies
\begin{align*}
\iota(w)^{-1}({q_{i}}-1)>\overline{p_{i}}\geqslant \iota(w)^{-1}({q_{i}}),
\end{align*}
for any $i$. This proves that $(q_{i}-1,\overline{p_{i}})$ satisfies Equation \eqref{eq:defconerSn}, which proves item (1).

For item (2), first of all, observe that every entry $x$ placed at position $z$ in Step $(i)$ satisfies $p_{i}\leqslant z$ and $x\leqslant \overline{q_{i}}$, implying that the correspondent dot at box $(\overline{x},\overline{z})$ in the diagram belongs to $\llr(p_{i},q_{i})$.

Now, we need to check that all dots placed by Step $(i)$ are outside $\llr(p_{i-1},q_{i-1})$. It is enough to verify that whenever in Step $(i)$ we are placing an entry $x$ at a position $z\geqslant p_{i-1}$ in the permutation, then $x>\overline{q_{i-1}}$. Set $\beta=(k_{i}-k_{i-1})-(p_{i-1}-p_{i})$ the number of entries to be placed after the position $p_{i-1}$ during Step $(i)$. If $1\leqslant i < a$ then condition B1$'$ implies that $\beta<q_{i-1}-q_{i}$. The entries that will be placed are $\overline{q_{i}+\beta}+1,\dots,\overline{q_{i}}$ and they are all strictly greater than $\overline{q_{i-1}}$ (in the diagram, it is equivalent to say that we have $q_{i-1}-q_{i}$ available rows to place the dots above $q_{i-1}$ but we only need $\beta$ rows). If $i=a$ then by Lemma \ref{lema:posEntriesStep}, $x>0>\overline{q_{i-1}}$.

If $a<i\leqslant s+1$ then condition B2$'$ implies that $\beta<(q_{i-1}-q_{i})-(k_{R(i-i)}-k_{R(i)})$, which means that have $(q_{i-1}-q_{i})-(k_{R(i-1)}-k_{R(i)})$ available rows in the diagram to place the dots above $q_{i-1}$ but we only need $\beta$ rows. Observe that we must skip $k_{R(i-1)}-k_{R(i)}$ rows in the diagram since their reflection have already been used between Steps $(R(i)+1)$ to $(R(i-1))$. This proves item (2).

Finally for (3), $k_{i}$ is the total of dots placed until Step $(i)$ and they are all placed inside the region $\Lambda(p_{i},q_{i})$. Any other dot placed after this step is placed outside $\Lambda(p_{i},q_{i})$.
\end{proof}

If we denote $\ttau$ as the set  $\{(k_{i},p_{i},q_{i}) \tq 1\leqslant i \leqslant s\}$, then Proposition \ref{prop:SEcorners} tell us that $\ttau$ as a subset of corners, i.e., we can denote $\ttau\subset \cnr(w)$.

Remember that there is a poset ``$<$'' in the set of corners $\cnr(w)$ where two corners positions satisfy $(p,q)<(p',q')$ if and only if $p>p'$ and $q<q'$. Also remember that the NE path $\NE(w)\subset\cnr(w)$ is the set of minimal elements of this poset.

\begin{lem}\label{lema:properties}
Let $w=w(\ttau)$ be a theta-vexillary signed permutation, and $\ttau=(\kk,\pp,\qq)$ be a theta-triple. Then every corner position $(p_{i},q_{i})$ of $\ttau$ is minimal in the poset ``$<$'', i.e., $\ttau\subset\NE(w)$.

\end{lem}
\begin{proof}
Suppose that there is a pair $(p_{i},q_{i})$ of $\ttau$ and a corner position $(p,q)\in\cnr(w)$ such that $(p,q)<(p_{i},q_{i})$, i.e., $p>p_{i}$ and $q<q_{i}$. The pair $(p,q)$ is not in $\ttau$ because $\pp$ and $\qq$ are strictly decreasing $s$-tuples. Since the box $(q-1,\overline{p})$ is a SE corner, Equation \eqref{eq:defconerSn} implies that
\begin{align}\label{eq:lema_minimal}
\overline{q}<x \ \ \mbox{ and }\ \  p\leqslant y,
\end{align}
where $x:=w(p-1)$ and $y:=w^{-1}(\overline{q})$.

When we use the construction algorithm to produce the permutation $w$, observe that the position $p-1$ must be filled by some step and the entry $\overline{q}$ must be placed in some step. So, there must be integers $1\leqslant m,l \leqslant s+1$ such that:
\begin{enumerate}
\item[a)] The entry $x$ is placed in the position $p-1$ during some Step $(m)$. This places a dot at the box $(\overline{x},\overline{p}+1)\in\Lambda(p_{m},q_{m})$;
\item[b)] The entry $\overline{q}$ is placed in the position $y$ during some Step $(l)$. This places a dot at the box $(q,\overline{y})\in\Lambda(p_{l},q_{l})$.
\end{enumerate}

Although there exist such integers $m$ and $i$, we are going to show that they cannot be either equal, smaller or greater than each other. Hence, this contradicts the assumption that $(p_{i},q_{i})$ is not minimal in the poset.

If $m=l$ then, using Equation \eqref{eq:lema_minimal}, $p-1<y$ are positions in Step $(m=l)$ and the entry in such positions are $w(p-1)=x >\overline{q}=w(y)$, i.e., there is a descent in it. This contradicts the fact that there are no descents in a step.

If $m<l$ then, using Equation \eqref{eq:lema_minimal}, we got that $\overline{y}\leqslant \overline{p}$ and $\overline{q} \leqslant x\leqslant \overline{q_{m}}$ (the former relation comes from the fact that every entry placed by Step $(m)$ is weakly smaller than $\overline{q_{m}}$). This implies that the box $(q,\overline{y})$ also belongs to the region $\Lambda(p_{m},q_{m})$, a contradiction of item 2 of Proposition \ref{prop:SEcorners}.

If $m>l$ then observe that Step $(l)$ must fill all positions from $p_{l}$ to $y$ in the  construction algorithm of the permutation $w$. Since $y>p-1\geqslant p_{i}\geqslant p_{l}$ (because $i<l$), we have that the position $p-1$ is also filled by Step $(l)$, which contradicts the fact that it is filled during Step $(m)$.
\end{proof}

Recall that a corner position $(p,q)$ of $\cnr(w)$ is unessential if there are corners $(p_{1},q_{1})$, $(p_{2},q_{2})$ and $(p_{3},q_{3})$ in the NE path $\NE(w)$ such that $(p,q)$ is not a minimal corner in the poset in the upper half of the diagram, the box $(q_{1}-1,\overline{p_{1}})$ lies above and in the same column of the box $(q-1,\overline{p})$, and the box $(\overline{q_{2}},p_{2}-1)$ reflected from $(q_{2}-1,\overline{p_{2}})$ lies to the right and in the same row of $(q-1,\overline{p})$, as in Figure \ref{fig:configuness}.

\begin{prop}\label{prop:svexcore}
Given $w\in \weyl_{n}$,
suppose that the set of corner $\cnr(w)$ is the disjoint union
$$
\cnr(w)=\NE(w)\dot{\cup}\uness(w).
$$
Then $w$ is a theta-vexillary.
\end{prop}
\begin{proof}

Suppose that the set of corners $\cnr(w)$ of a permutation $w$ is given by the disjoint union of the NE path $\NE(w)$ and the set of unessential corners $\uness(w)$. Since all corner positions $(p_{i},q_{i})$ of $\NE(w)$ can be ordered so that $p_{1}\geqslant p_{2}\geqslant \cdots \geqslant p_{r}>0$ and $q_{1}\geqslant q_{2}\geqslant \cdots \geqslant q_{r}$, set $k_{i}$ as the rank $r_{w}(p_{i},q_{i})$ and define the triple $\ttau'=(\kk,\pp,\qq)$. We will prove that $\ttau$ is almost a theta-vexillary triple, i.e., it satisfies A1, A2, A3, C1, C2, and B1. In order to get B2 and B3, occasionally some elements $(k_{i},p_{i},q_{i})$ should be removed from $\ttau'$.

Conditions A1, A2 and A3 are true because $w$ is a signed permutation in $\weyl_{n}$. In fact, A1 and A3 come direct from the fact that there is no SE corner at row $-1$ or above the middle in column $-1$ since $w(0)=0$, and A2 is satisfied just because we cannot have dots lying in opposite rows.

Now, $a$ and $R(i)$, for $a\leqslant i\leqslant s$, can be defined. Let us prove that $\ttau$ satisfies the remaining conditions. Consider the diagrams sketched in Figure \ref{fig:condproofBC}.

\begin{figure}[ht]
	\centering
	\includegraphics[scale=0.8]{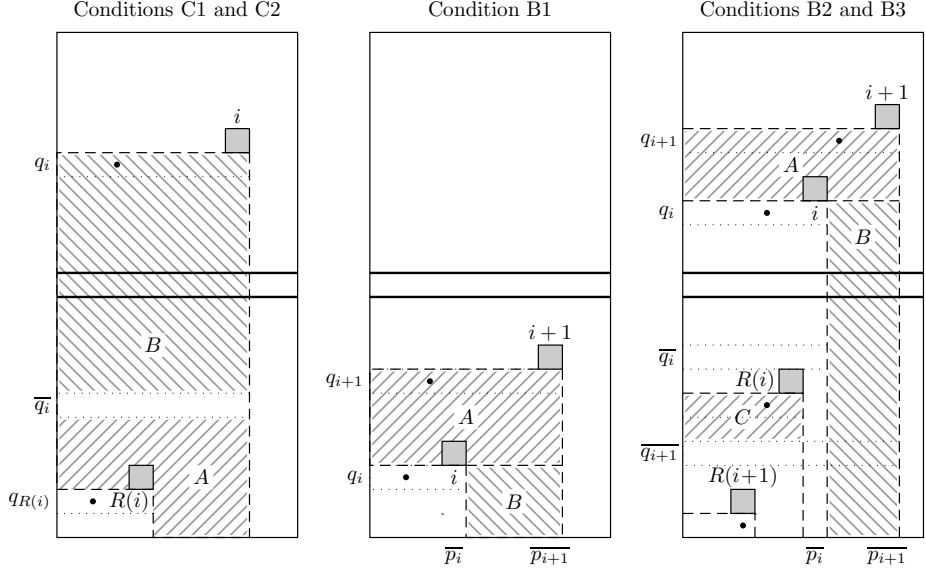}
	\caption{Configuration required to get conditions C1, C2 (left), B1 (middle), B2, and B3 (right).}
	\label{fig:condproofBC}
\end{figure}

For condition C1, let $a\leqslant i\leqslant s$ and consider the regions $A$ and $B$ as in Figure \ref{fig:condproofBC} (left). Denote by $d(A)$ and $d(B)$ the number of dots in each one of them.
The definition of $R(i)$ can be translated to the diagram as follows: $R(i)$ is the unique index smaller than $a$ such that there is no other corner of $\ttau$ lying to the right of it and in the rows $q_{R(i)}-1,\dots, \overline{q_{i}}$.
Suppose that there is a dot in the darker region $A$ of Figure \ref{fig:condproofBC} (left). This dot must be placed by some Step $(j)$, for $j>R(i)$, which implies that the corner position $(p_{j},q_{j})$ is located above the row $\overline{q_{i}}$ and it also places a dot above $\overline{q_{i}}$. However, the construction of a step says that we must fill all entries between them, including $q_{i}$. So, we should have a dot at row $q_{i}$ and another in the row $\overline{q_{i}}$, a contradiction of condition A2. Hence, $d(A)=0$. On the other hand, $d(B)\leqslant -q_{i}$ because condition A2 says that we cannot have dots in opposite rows. Thus, $-q_{i}\geqslant d(A)+d(B)=k_{i}-k_{R(i)}$, since $d(A)+d(B)$ is the amount of dots to be placed from Step $(R(i)+1)$ to $(i)$.

By Lemma \ref{lema:posEntriesStep}, we may show that $\ttau$ satisfies condition C2$'$ instead of C2. In the previous case, we proved that region $A$ contains no dots. It means that no step from $(R(i)+1)$ to $(a-1)$ place dots in $A$, which is equivalent to say that all entries placed by Steps $(R(i)+1)$ to $(a-1)$ are strictly bigger than $q_{i}$, proving C2$'$.

For condition B1, let $1\leqslant i < a-1$ and consider the rectangular regions $A$ and $B$ as in Figure \ref{fig:condproofBC} (middle). Denote by $d(A)$ and $d(B)$ the number of dots in each one of them. Notice that the number of dots in each rectangle is limited by the length of their sides, and $d(A)+d(B)$ is the number of dots in Step $(i)$. If $p_{i}=p_{i+1}$ then $d(B)=0$ and $d(A)<q_{i}-q_{i+1}$, since we cannot place a dot in row ${q_{i}}-1$. So, $(p_{i}-p_{i+1})+(q_{i}-q_{i+1})> d(B)+d(A)=k_{i+1}-k_{i}$. If $p_{i}>p_{i+1}$ then, we cannot have the dot in column $\overline{p_{i}}+1$ inside $B$ because it would not create a SE corner $(q_{i}-1,\overline{p_{i}})$. Hence, $(p_{i}-p_{i+1})+(q_{i}-q_{i+1})> d(B)+d(A)=k_{i+1}-k_{i}$.

For conditions B2 and B3, let $a\leqslant i\leqslant s$ and consider the rectangular regions $A$, $B$ and $C$ of Figure \ref{fig:condproofBC} (right). Suppose that $p_{i}>p_{i+1}$. Using the same argument of condition C1, all dots between the rows $\overline{q_{i}}$ and $\overline{q_{i+1}}$ are in rectangle $C$ and the number of dots in this region is $d(C)=k_{R(i)}-k_{R(i+1)}$. As well as the previous case, $d(B)< p_{i}-p_{i+1}$ and the number of dots in region $A$ is $d(A)\leqslant (q_{i}-q_{i+1})-d(C)$, since we cannot have dots in opposite rows. Therefore, $(p_{i}-p_{i+1})+(q_{i}-q_{i+1})> d(B)+d(A)+d(C)=(k_{i+1}-k_{i})+(k_{R(i)}-k_{R(i+1)})$.

The difficulty appears when $p_{i}=p_{i+1}$. In this case, $d(B)=0$ and $d(A)\leqslant (q_{i}-q_{i+1})-d(C)$. Then, $(p_{i}-p_{i+1})+(q_{i}-q_{i+1})\geqslant d(B)+d(A)+d(C)=(k_{i+1}-k_{i})+(k_{R(i)}-k_{R(i+1)})$, which means that the equality can happen. So, we need to remove these elements from $\ttau'$ where the equality holds. Denote the set of index $I=I_{\ttau'}\subset [1,s]$ by
\begin{align*}
I=\{i\geqslant a\tq 
(p_{i}-p_{i+1})+(q_{i}-q_{i+1})=(k_{i+1}-k_{i})+(k_{R(i)}-k_{R(i+1)})\}
\end{align*}

Define $\ttau$ the triple
\begin{align*}
\ttau := \{(k_{i},p_{i},q_{i})\in \ttau' \tq i\not\in I\}
\end{align*}

Clearly, $\ttau$ satisfies A1, A2, A3, C1, C2, and B1. Suppose that $a\leqslant i<j$ are integers such that $i,j\not\in I$ and $i+1,i+2,\dots, j-1 \in I$, i.e., $i$ and $j$ are consecutive indexes in $\ttau$. Then they satisfy
\begin{gather*}
(p_{i}-p_{i+1})+(q_{i}-q_{i+1}) > (k_{i+1}-k_{i})+(k_{R(i)}-k_{R(i+1)}),\\
(p_{i+1}-p_{i+2})+(q_{i+1}-q_{i+2}) = (k_{i+2}-k_{i+1})+(k_{R(i+1)}-k_{R(i+2)}),\\
(p_{i+2}-p_{i+3})+(q_{i+2}-q_{i+3}) = (k_{i+3}-k_{i+2})+(k_{R(i+2)}-k_{R(i+3)}),\\
\vdots \\
(p_{j-1}-p_{j})+(q_{j-1}-q_{j}) = (k_{j}-k_{j-1})+(k_{R(j-1)}-k_{R(j)}).\\
\end{gather*}

Therefore,
\begin{align*}
(p_{i}-p_{j})+(q_{i}-q_{j}) > (k_{j}-k_{i})+(k_{R(i)}-k_{R(j)}),\\
\end{align*}
and $\ttau$ also satisfies B2 and B3.

Finally, observe that the extended diagram of $w(\ttau)$ is exactly the extended diagram of $w$, which means that $w(\ttau)=w$.
\end{proof}

Now, we aim to proof the converse of Proposition \ref{prop:svexcore}.

\begin{lem}\label{lema:properties2}
Let $w=w(\ttau)$ be a theta-vexillary signed permutation, and $\ttau=(\kk,\pp,\qq)$ be a theta-triple. Then for any $1\leqslant i\leqslant s$ such that $p_{i}>p_{i+1}$, there is no corner position $(p,q)$ different of $(p_{i},q_{i})$ satisfying $p>p_{i+1}$ and $q_{i}\geqslant q$. In other words, $(p_{i},q_{i})$ is the unique SE corner in the region highlighted in Figure \ref{fig:specialSE}.

\begin{figure}[ht]
	\centering
	\includegraphics[scale=0.8]{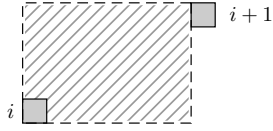}
	\caption{Region in the extended diagram where we cannot have a SE corner.}
	\label{fig:specialSE}
\end{figure}
\end{lem}
\begin{proof}
Suppose that there is $(p,q)$ for some $i$ such that $p>p_{i+1}$. If $p_{i}>p>p_{i+1}$ then the position $p-1$ is a descent of $w$, which is impossible since all descents of $w$ are at positions $p_{i}-1$ and no one matches to $p-1$.

If $p_{i}=p$ and $q>q_{i}$ then the box $(q-1,\overline{p_{i}})$ is a SE corner, and the dots in row $\overline{q}$ and column $\overline{p_{i}+1}$ are placed as in Figure \ref{fig:prooflema1}.
The dot placed in row $\overline{q}$ lies inside the region $\Lambda(p_{i+1},q_{i+1})$, and outside $\Lambda(p_{i},q_{i})$, implying that such dot is placed during Step $(i+1)$.
\begin{figure}[ht]
	\centering
	\includegraphics[scale=0.8]{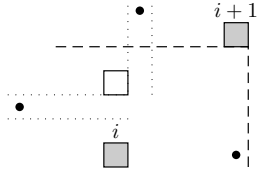}
	\caption{Sketch of the diagram to prove Lemma \ref{lema:properties2}.}
	\label{fig:prooflema1}
\end{figure}

Notice that the dot at column $\overline{p_{i}}+1$ cannot be placed during Step $(i+1)$ because it would create a descent in Step $(i+1)$. Then, there is $j>i+1$ such that  Step $(j)$ placed such dot. In this case, Step $(i+1)$ should skip column $\overline{p_{i}}+1$, which is impossible (by construction, this step places dots in all available columns between $w^{-1}(\overline{q}$ and $p_{i+1}$).
\end{proof}

The NE path also can contain another kind of SE corner defined as follows: given a theta-vexillary signed permutation $w$ and a theta-triple $\ttau$, a corner position $(p,q)\not\in\ttau$ is called \emph{optional} if there are $a\leqslant i\leqslant s$ and $1\leqslant j<a$ such that $p=p_{i}$, $q_{i}<q=\overline{q_{j}}+1$ and $q_{i-1}\geqslant q>q_{i}$. In other words, $(p,q)$ belongs to the NE path just between the corners $(p_{i-1},q_{i-1})$ and $(p_{i},q_{i})$, and the box $(q_{i}-1,\overline{p_{i}})$ lies above and in the same column of $(q-1,\overline{p})$, as shown in Figure \ref{fig:configopt}. Denote by $\op(w)$ the set of all optional corners and observe that $\op(w)\subset\NE(w)$.

\begin{figure}[ht]
	\centering
	\includegraphics[scale=0.8]{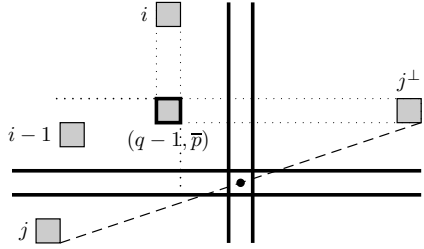}
	\caption{Configuration of an optional corner $(p,q)$.}
	\label{fig:configopt}
\end{figure}

Observe that such box only occurs if the number of available rows between $q_{i}$ and $q$ is smaller than the number of dots to be placed by Step $(i)$, which is $k_{i}-k_{i-1}$. In other words, we need to have enough dots to place during Step $(i)$ such that some of them are placed below the corner $(p,q)$. This implies that the following equation is satisfied:
\begin{align}\label{eq:optional}
q-q_{i}=k_{i}-k+k_{j}-k_{R(i)}.
\end{align}

Thus, a triple $\ttau'$ obtained by adding $(k,p,q)$ to $\ttau$ also gives the same permutation but it is not a theta-triple anymore.

For instance, considering the theta-vexillary signed permutation $w=10\ 1\ 5\allowbreak\ 3\ \overline{2}\ \overline{4}\ 6\ \overline{9}\ \overline{8}\ \overline{7}$ of Figure \ref{fig:diagram1}, the set of optional corners $\op(w)$ only contains the triple $(7,2,\overline{3})$.

\begin{prop}\label{prop:svexcore3}
Let $w$ be a theta-vexillary and $\ttau$ be a theta-triple. Then, the set of corners is the disjoint union
$$
\cnr(w)=\ttau \ \dot{\cup}\ \op(w) \ \dot{\cup}\ \uness(w).
$$
\end{prop}
\begin{proof}

Denote by $\iota(\ttau)\subset \cnr(\iota(w))$ the set of all corner positions of $\ttau$ and their reflections, i.e.,
\begin{align*}
\iota(\ttau)=\bigcup_{i=1}^{s}\{(k_{i},p_{i},q_{i}) \cup (k_{i},p_{i},q_{i})^{\perp}\}.
\end{align*}
 
Propositions \ref{prop:descentsofW} and \ref{prop:descentsofWinv} state that all descents of $w$ and $w^{-1}$ are exclusively determined by elements in $\ttau$. More over, this assertion can be extended to the diagram $D(\iota(w)))$ of $\iota(w)$: all descents of $\iota(w)$ are at position $p_{i}-1$ and $\overline{p_{i}}$, and all descents of $\iota(w)^{-1}$ are at position $q_{i}-1$ and $\overline{q_{i}}$, where $i$ ranges from $1$ to $s$. Thus, if there is other SE corner, it should not create descents, but it must match existing descents. For instance, for $\ttau$ of Example \ref{ex:perm1} and its diagram in Figure \ref{fig:diagram1}, observe that the corner position $(2,\overline{3})$ does not belong to $ \iota(\ttau)$ but it is in the same row and column of two corner positions of $\iota(\ttau)$, namely, $(4,\overline{3})$ and $(2,\overline{6})$.

We conclude that if there exists a corner position $T$ that does not belong to $\iota(\ttau)$ then there are corner positions $T',T''$ in $ \iota(\ttau)$ such that $T'$ is in the same column of $T$, and $T''$ is in the same row of $T$.
Then,  we need to figure out when a combination of descents of corners $T'$ and $T''$ in $\iota(\ttau)$ a new corner.

Consider the diagram $D(\iota(w))$ divided in quadrants as in Figure \ref{fig:quadrant}. 

\begin{figure}[ht]
	\centering
	\includegraphics[scale=0.8]{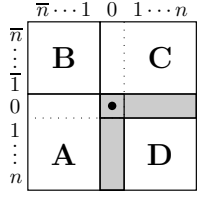}
	\caption{Quadrants of the diagram.}
	\label{fig:quadrant}
\end{figure}

Given $(p,q)\in\iota(\ttau)$, we say that the SE corner $(q-1,\overline{p})$ belongs to:
\begin{itemize}
\item Quadrant \textbf{A} if $(p,q)=(p_{i},q_{i})$ for some $i<a$;
\item Quadrant \textbf{B} if $(p,q)=(p_{i},q_{i})$ for some $i\geqslant a$;
\item Quadrant \textbf{C} if $(p,q)=(p_{i},q_{i})^{\perp}$ for some $i< a$;
\item Quadrant \textbf{D} if $(p,q)=(p_{i},q_{i})^{\perp}$ for some $i\geqslant a$.
\end{itemize}

Consider two corner positions $T'=(p',q')$ and $T''=(p'',q'')$ in $\iota(\ttau)$ such that $p'>p''$ and $q'\neq q''$, i.e., $T'$ and $T''$ are in different rows and columns. A \emph{cross descent} is a box that lies in the same row of one of these corners (either $T'$ or $T''$) and in the same column of the remaining one. There are four types of cross descent boxes of $T'$ and $T''$ as shown in Figure \ref{fig:crossing}.
\begin{figure}[ht]
	\centering
	\includegraphics[scale=0.8]{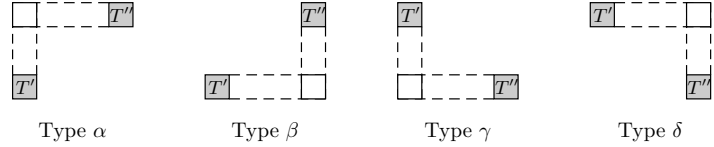}
	\caption{Four possibles cross descents boxes.}
	\label{fig:crossing}
\end{figure}

Namely, 
\begin{itemize}
\item \emph{A cross descent of type $\alpha$} is the box $(q''-1,\overline{p'})$ when $q'>q''$;
\item \emph{A cross descent of type $\beta$} is the box $(q'-1,\overline{p''})$ when $q'>q''$;
\item \emph{A cross descent of type $\gamma$} is the box $(q''-1,\overline{p'})$ when $q'<q''$;
\item \emph{A cross descent of type $\delta$} is the box $(q'-1,\overline{p''})$ when $q'<q''$.
\end{itemize}

Suppose that $T'=(p',q')$ and $T''=(p'',q'')$ are two corners of $\iota(\tau)$. Consider that $T'$ lies in some quadrant $\mathbf{X}$, $T''$ lies in some quadrant $\mathbf{Y}$, and they have cross descent box $(a,b)$ of type $\xi$, where $\mathbf{X},\mathbf{Y}\in \{\mathbf{A},\mathbf{B},\mathbf{C},\mathbf{D}\}$  and $\xi\in\{\alpha,\beta,\gamma,\delta\}$. We say that this configuration has \emph{shape} \shape{X}{\xi}{Y}. Also denote by $c_{\xi}(T',T'')=(a,b)$ the respective cross descent box. 

First of all, we need to figure out all possible shapes and, then, verify if the cross descent box of each shapes is an SE corner.

There are $64$ different combination of shapes \shape{X}{\xi}{Y}, where $\mathbf{X},\mathbf{Y}\in \{\mathbf{A},\mathbf{B},\mathbf{C},\mathbf{D}\}$  and $\xi\in\{\alpha,\beta,\gamma,\delta\}$. However, not every shape is possible because $\ttau$ is a theta-triple and $T',T''$ are chosen in $\iota(\ttau)$. An example of impossible shape is \shape{A}{\delta}{A} since, by definition, there is no $i<j$ where $T'=(p_{i},q_{i})$ and $T''=(p_{j},q_{j})$ such that $q_{i}<q_{j}$. Thus, it remains only $24$ possible shapes. We listed them in Table \ref{tab:shapes}, divided in two categories: the shapes $\shape{X}{\xi}{Y}$ where the cross descent box $c_{\xi}(T',T'')$ belongs to the quadrants $\mathbf{A}$ or $\mathbf{B}$, and the shapes $\shape{X}{\xi}{Y}$ where $c_{\xi}(T',T'')$ belongs to the quadrants $\mathbf{C}$ of $\mathbf{D}$.

\begin{table}[ht]
\centering
\caption{Possible shapes}
\label{tab:shapes}
\begin{tabular}{|c|c|}
\hline
Shapes $\shape{X}{\xi}{Y}$ where $c_{\xi}(T',T'')$ & Shapes $\shape{X}{\xi}{Y}$ where $c_{\xi}(T',T'')$ \\
belongs to \textbf{A} or \textbf{B}: & belongs to \textbf{C} or \textbf{D}: \\
\hline
\shape{A}{\alpha}{A}, \shape{A}{\alpha}{B}, \shape{A}{\alpha}{C}, &
\shape{C}{\beta}{C}, \shape{D}{\beta}{C}, \shape{A}{\beta}{C}, \\
\shape{A}{\alpha}{D}, \shape{B}{\alpha}{B}, \shape{B}{\alpha}{C},&
\shape{B}{\beta}{C}, \shape{D}{\beta}{D}, \shape{A}{\beta}{D},\\
\shape{A}{\beta}{A}, \shape{A}{\beta}{B}, \shape{B}{\beta}{B}, &
\shape{C}{\alpha}{C}, \shape{D}{\alpha}{C}, \shape{D}{\alpha}{D},\\
\shape{A}{\gamma}{D}, \shape{B}{\gamma}{C}, \shape{B}{\gamma}{D}. &
\shape{B}{\delta}{C}, \shape{A}{\delta}{D}, \shape{B}{\delta}{D}.\\ 
\hline
\end{tabular}
\end{table}

Observe that if $c_{\xi}(T',T'')$ belongs to quadrants \textbf{C} or \textbf{D} then its reflection $c_{\xi}(T',T'')^{\perp}$ belongs to quadrant \textbf{A} or \textbf{B} and corresponds to the cross descent box of corners $(T'')^{\perp}=(\overline{p''}+1,\overline{q''}+1)$ and $(T')^{\perp}=(\overline{p'}+1,\overline{q'}+1)$. In other words, each shape in the left column of Table \ref{tab:shapes} is equivalent to another one to the right column. Hence, we can consider only the $12$ shapes where $c_{\xi}(T',T'')$ belongs to quadrants \textbf{A} or \textbf{B}.

It follows from Lemma \ref{lema:properties} that  $\ttau\cap\uness(w)=\varnothing$ because no unessential corner is minimal in the poset. By definition of optional corner, we also have that $\ttau\cap\op(w)=\varnothing$ and $\op(w) \cap \uness(w)=\varnothing$. Then, all the sets are disjoint.

Suppose that $T'=(p',q')$ and $T''=(p'',q'')$ of $\iota(\tau)$ has some shape \shape{X}{\xi}{Y}, where $\mathbf{X,Y}\in \{\mathbf{A,B,C,D}\}$ and $\xi\in\{\alpha,\beta,\gamma,\delta\}$, such that the cross descent box $c_{\xi}(T',T'')$ is a SE corner in quadrant \textbf{A} or \textbf{B} which \emph{does not} belongs to $\ttau$.
Then, analyzing each situation in the first column of Table \ref{tab:shapes}, we must show that either $c_{\xi}(T',T'')\in\ \op(w) \ \dot{\cup}\ \uness(w)$ or it leads us to a contradiction.

Consider $\xi=\alpha$, where $p'>p''$, $q'>q''$, and $T=(p',q'')$ is a SE corner $(q''-1,\overline{p'})$ not in $\ttau$ and satisfying the following conditions
\begin{align}\label{eqn:SEcorneralpha}
\begin{split}
\iota(w)(p'-1)>\overline{q''} &\geqslant \iota(w)(p')\\
\iota(w)^{-1}(\overline{q''})>p'-1 &\geqslant \iota(w)^{-1}(\overline{q''}+1).
\end{split}
\end{align}

\begin{itemize}
\item If \shape{X}{\xi}{Y} is a shape \shape{A}{\alpha}{A}, \shape{A}{\alpha}{B} or \shape{B}{\alpha}{B}, then $T'=(p_{i},q_{i})$, $T''=(p_{j},q_{j})$, where $1\leqslant i<j\leqslant s$, and $T=(p_{i},q_{j})$ is a SE corner $(q_{j}-1,\overline{p_{i}})$. But Lemma \ref{lema:properties2} says that $T$ cannot be a corner.

\item If \shape{X}{\xi}{Y} is a shape \shape{A}{\alpha}{C} or \shape{B}{\alpha}{C}, then $T'=(p_{i},q_{i})$, for some $i$, $T''=(p_{j},q_{j})^{\perp}=(\overline{p_{j}}+1,\overline{q_{j}}+1)$, for some $j<a$, and $q_{i}>\overline{q_{j}}+1$.
We can assume that $i$ is chosen such that there is no $l>i$ satisfying $p_{i}=p_{l}$ and $q_{i}>q_{l}>\overline{q_{j}}+1$, i.e., there is no corner of $\ttau$ in the same column and between the SE corners $T'$ and $T$. If $p_{i}>p_{i+1}$ then Lemma \ref{lema:properties2} is contradicted. Thus, we have that $p_{i}=p_{i+1}$ and $q_{i}>\overline{q_{j}}+1>q_{i+1}$, implying that $T$ is an optional SE corner.

\item If \shape{X}{\xi}{Y} is a shape \shape{A}{\alpha}{D}, then $T'=(p_{i},q_{i})$, for some $i<a$, $T''=(p_{j},q_{j})^{\perp}=(\overline{p_{j}}+1,\overline{q_{j}}+1)$, for some $a\leqslant j\leqslant s$, and $q_{i}>\overline{q_{j}}+1$. As in the previous case, we can assume that $i$ is chosen such that there is no $l>i$ satisfying $p_{i}=p_{l}$ and $q_{i}>q_{l}>\overline{q_{j}}+1$, i.e., there is no corner of $\ttau$ in the same column and between the SE corners $T'$ and $T$. If 
$p_{i}>p_{i+1}$, then the corner $T$ contradict Lemma \ref{lema:properties2}. Thus, $p_{i}=p_{i+1}$, $q_{i}> \overline{q_{j}}+1> q_{i+1}$ and $i=R(j)$. Notice that the dot in the row $\overline{q_{j}}+1$ is between rows $q_{i}$ and $\overline{q_{j}}$ since $T$ is a corner, which is impossible as shown in proof of Proposition \ref{prop:svexcore} (see Figure \ref{fig:proofprop2}).
\begin{figure}[ht]
	\centering
	\includegraphics[scale=0.8]{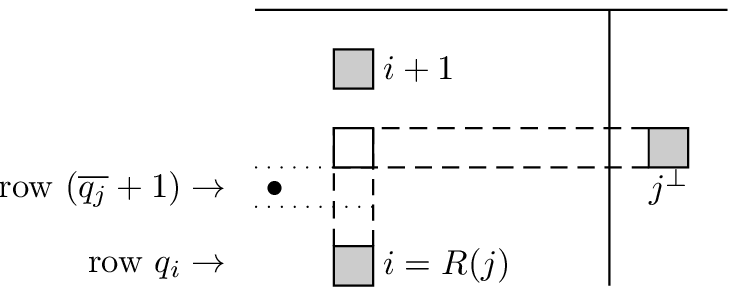}
	\caption{Sketch of the diagram of shape \shape{A}{\alpha}{D}.}
	\label{fig:proofprop2}
\end{figure}
\end{itemize}

Consider $\xi=\beta$, where $p'>p''$, $q'>q''$, and $T=(p'',q')$ is a SE corner $(q'-1,\overline{p''})$ is a SE corner not in $\ttau$ and satisfying the following conditions
\begin{align}\label{eqn:SEcornerbeta}
\begin{split}
\iota(w)(p''-1)>\overline{q'} &\geqslant \iota(w)(p'')\\
\iota(w)^{-1}(\overline{q'})>p''-1 &\geqslant \iota(w)^{-1}(\overline{q'}+1).
\end{split}
\end{align}

\begin{itemize}
\item If \shape{X}{\xi}{Y} is a shape \shape{A}{\beta}{A} or \shape{A}{\beta}{B}, then $T'=(p_{i},q_{i})$ and $T''=(p_{j},q_{j})$, for some $i<a$ and $i<j$.
Observe that the dots at column $\overline{p_{j}}$ and row $q_{j}$ are placed by Step $(j)$ (or some previous one). Then, by construction, the dot at row $q_{i}-1$ must be placed by some Step $(l)$ for $l\leqslant j$. Thus, $\iota(w)^{-1}(q_{i}-1)\geqslant \overline{p_{j}}$, a contradiction of Equation \eqref{eqn:SEcornerbeta}.

\item If \shape{X}{\xi}{Y} is a shape \shape{B}{\beta}{B}, then $T'=(p_{i},q_{i})$ and $T''=(p_{j},q_{j})$, for some $a\leqslant i<j\leqslant s$.
If $\iota(w)^{-1}({q_{i}}-1)<0$, i.e., the dot in the row $q_{i}-1$ is in quadrant \textbf{B} then we can proceed as the previous case. If $\iota(w)^{-1}({q_{i}}-1)>0$ then 
belongs to the quadrant \textbf{C} is a reflection of a dot placed during some Step $(l)$ for $l<a$. Since $\iota(w)(q_{i})<\overline{p_{j}}+1\leqslant 0$, then $q_{l}=\overline{q_{i}}+1$ and the corner $(p_{l},q_{l})$ lies in row $\overline{q_{i}}+1$. Therefore, the reflection $(p_{l},q_{l})^{\perp}$ is in the row $q_{i}-1$, and the corner $T$ is optional or unessential (see Figure \ref{fig:proofprop4}).
\begin{figure}[ht]
	\centering
	\includegraphics[scale=0.8]{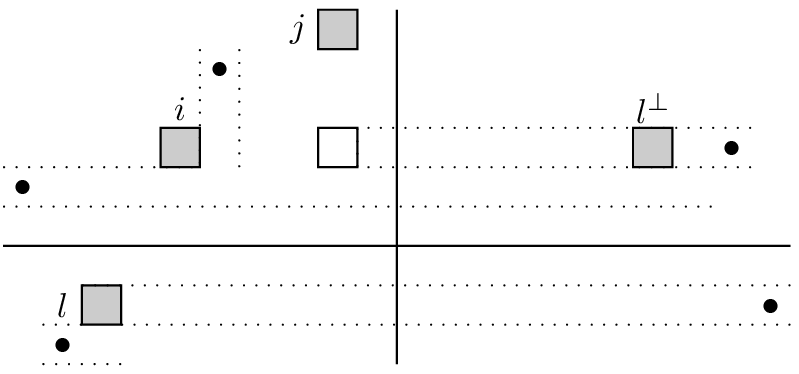}
	\caption{Sketch of the diagram of shape \shape{B}{\beta}{B}.}
	\label{fig:proofprop4}
\end{figure}
\end{itemize}

Consider $\xi=\gamma$, where $p'>p''$, $q'<q''$, and $T=(p',q'')$ is a SE corner $(q''-1,\overline{p'})$ is a SE corner not in $\ttau$ and satisfying the following conditions
\begin{align}\label{eqn:SEcornergamma}
\begin{split}
\iota(w)(p'-1)>\overline{q''} &\geqslant \iota(w)(p')\\
\iota(w)^{-1}(\overline{q''})>p'-1 &\geqslant \iota(w)^{-1}(\overline{q''}+1).
\end{split}
\end{align}

\begin{itemize}
\item If \shape{X}{\xi}{Y} is a shape \shape{A}{\gamma}{D} or \shape{B}{\gamma}{D}, then $T'=(p_{i},q_{i})$, for any $i$, $T''=(p_{j},q_{j})^{\perp}=(\overline{p_{j}}+1,\overline{q_{j}}+1)$, for some $a\leqslant j\leqslant s$, and $q_{i}<\overline{q_{j}}+1$.
By Equation \eqref{eqn:SEcornergamma}, $q_{j}-1 \geqslant \iota(w)(p_{i})$ and $\iota(w)^{-1}(q_{j}-1)>p_{i}-1\geqslant 0> \iota(w)^{-1}(q_{j})$, implying that no Step $(l)$, for $l<a$, can place the dot at row $\overline{q_{j}}$. Hence, $q_{R(j)}=\overline{q_{j}}+1$ and $T$ is exactly the corner $(p_{R(j)},q_{R(j)})$ of $\ttau$ (see Figure \ref{fig:proofprop3}).
\begin{figure}[ht]
	\centering
	\includegraphics[scale=0.8]{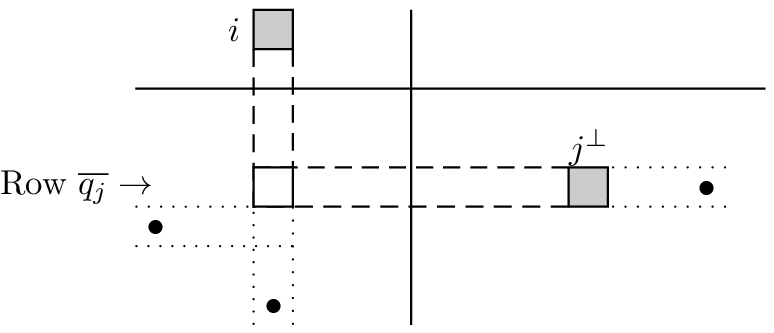}
	\caption{Sketch of the diagram of shape \shape{A}{\gamma}{D} or \shape{B}{\gamma}{D}.}
	\label{fig:proofprop3}
\end{figure}

\item If \shape{X}{\xi}{Y} is a shape \shape{B}{\gamma}{C}, then we clearly have that $T$ is an unessential or optional corner. \qedhere
\end{itemize}
\end{proof}

\begin{prop}\label{prop:svexcore2}
For $w\in \weyl_{n}$, $w$ is theta-vexillary if and only if the set of corner $\cnr(w)$ is the disjoint union
$$
\cnr(w)=\NE(w)\dot{\cup}\uness(w).
$$
\end{prop}
\begin{proof}
Suppose that $w(\ttau)$ is theta-vexillary. From Lemma \ref{lema:properties}, $\ttau\cup\op(w)\subset\NE(w)$. On the other hand, Proposition \ref{prop:svexcore3} implies that $\NE(w)\subset\ttau\cup\op(w)$ since $\NE(w)\cap\uness(w)=\varnothing$. Hence, $\NE(w)= \ttau\cup\op(w)$.
\end{proof}

\begin{rem}
If $w$ is a theta-vexillary signed permutation but we don't know a theta-triple such that $w=w(\ttau)$, we can use the process in the proof of Proposition \ref{prop:svexcore} to get $\ttau$. Basically, set $\ttau$ with all the corners in the NE path $\NE(w)$. Then, withdraw all the optional corners from it, which results in a valid theta-triple $\ttau$ of $w$.
\end{rem}

\begin{prop}\label{prop:unique}
The theta-triple is unique for each theta-vexillary signed permutation.
\end{prop}
\begin{proof}
Suppose that $\ttau$ and $\tilde\ttau$ are two different theta-triples such that $w=w(\ttau)=w(\tilde\ttau)$. Then, $\ttau\dot{\cup}\op(w)=\NE(w)=\tilde{\ttau}\dot{\cup} \mathcal{O}p_{\tilde\ttau}$. If there is a corner position $(p,q_{1})\in \mathcal{O}p_{\ttau}\cap \tilde{\ttau}$ then there is $q_{2}>q_{1}$ such that $(p,q_{2})\in \ttau$ is a corner position immediately above it. Notice that $(p,q_{2})$ does not belong to $\tilde\ttau$, otherwise condition B2 of $\tilde{\ttau}$ for both corners would contradict Equation \eqref{eq:optional} for the optional corner $(p,q_{1})$. Then, $(p,q_{2}) \in \mathcal{O}p_{\tilde\ttau}\cap {\ttau}$. For the same reason, there is $q_{3}>q_{2}$ such that $(p,q_{3})\in \mathcal{O}p_{\ttau}\cap \tilde{\ttau}$, and keep going. Hence, this process should be repeated forever, which is impossible since the sets are finite. Therefore, $\mathcal{O}p_{\ttau}\cap \tilde{\ttau}=\emptyset$, and by the same reason $\mathcal{O}p_{\tilde\ttau}\cap {\ttau}=\emptyset$, which implies that $\ttau=\tilde{\ttau}$.
\end{proof}

\section{Pattern avoidance}

Recall that given a signed pattern $\pi=\pi(1)\ \pi(2)\cdots \pi(m)$ in $\weyl_{m}$, a signed permutation $w$ \emph{contains} $\pi$ if there is a subsequence $w(i_{1})\cdots w(i_{m})$ such that the signs of $w(i_{j})$ and $\pi(j)$ are the same for all $j$, and also the absolute values of the subsequence are in the same relative order as the absolute values of $\pi$. Otherwise $w$ \emph{avoids} $\pi$.

\begin{prop}\label{prop:patavoid}
A signed permutation $w$ is theta-vexillary if and only if $w$ avoids the follow thirteen signed patterns
$[\overline{1}\ 3\ 2]$, 
$[\overline{2}\ 3\ 1]$, 
$[\overline{3}\ 2\ 1]$, 
$[\overline{3}\ 2\ \overline{1}]$, 
$[2\ 1\ 4\ 3]$, 
$[2\ \overline{3}\ 4\ \overline{1}]$, 
$[\overline{2}\ \overline{3}\ 4\ \overline{1}]$, 
$[3\ \overline{4}\ 1\ \overline{2}]$,
$[3\ \overline{4}\ \overline{1}\ \overline{2}]$,
$[\overline{3}\ \overline{4}\ 1\ \overline{2}]$, 
$[\overline{3}\ \overline{4}\ \overline{1}\ \overline{2}]$, 
$[\overline{4}\ 1\ \overline{2}\ 3]$,
and $[\overline{4}\ \overline{1}\ \overline{2}\ 3]$.
\end{prop}
\begin{proof}
We know, by Proposition \ref{prop:svexcore}, how to describe a theta-vexillary permutation by the SE corners of the extended diagram.

Assume that $w$ is a theta-vexillary signed permutation. To prove that it avoids all these 13 patterns, we will assume if one of these patterns is contained in $w$, then show that there is a SE corner $T$ such that $T\not\in\NE(w)\cup\uness(w)$.

Suppose that $w$ contains $[2\ 1\ 4\ 3]$ as a subsequence $(w(a)\ w(b)\ w(c)\ w(d))$ satisfying $w(b)<w(a)<w(d)<w(c)$ for some $a<b<c<d$ as in Figure \ref{fig:patternsavoid2143} (left). Then, there is at least one SE corner in each shaded area, which will be denoted by $T$ and $T'$. Clearly $T\not\in\NE(w)$ and, by Proposition \ref{prop:svexcore2}, it should be an unessential corner. Since $\ttau$ is a theta-triple of $w$, there are $i<j$ such that Step $(i)$ places the dot in the column $\overline{b}$ and Step $(j)$ place the dot in the column $\overline{a}$. However it lead us to a contradiction because we cannot place a dot in the row $w(\overline{b})$ during Step $(i)$ and skip the row $w(\overline{a})$ since it will be place further.

\begin{figure}[ht]
	\centering
	\includegraphics[scale=0.8]{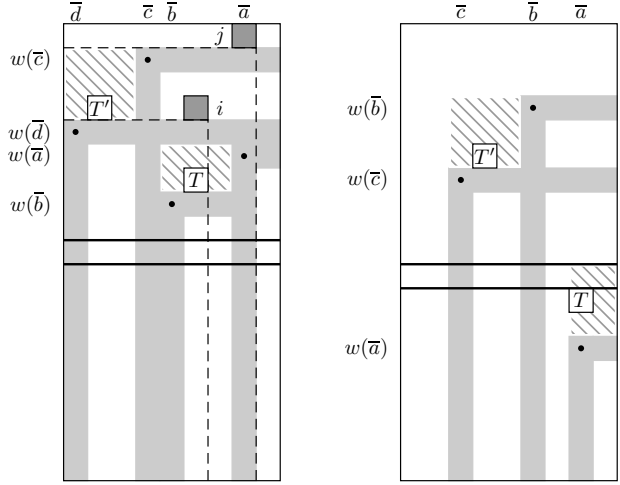}
	\caption{Situation where $w$ contains: $[2\ 1\ 4\ 3]$ in the left; and $[\overline{1}\ 3\ 2]$ in the right.}
	\label{fig:patternsavoid2143}
\end{figure}

Now, suppose that $w$ contains $[\overline{1}\ 3\ 2]$ as a subsequence $(w(a)\ w(b)\ w(c))$ satisfying $\overline{w(a)}<w(c)<w(b)$ for some $a<b<c$ as in Figure \ref{fig:patternsavoid2143} (right). Then, there is at least one SE corner in each shaded area, which will be denoted by $T$ and $T'$. By definition, $T$ is neither an unessential corner nor belongs to the NE path, i.e, $T\not\in \NE(w)\cup\uness(w)$, which contradicts Proposition \ref{prop:svexcore2}.

Notice that we could consider the diagram of $w$ restricted to the columns $\overline{a},\overline{b},\overline{c}$ and rows $0,\pm w(\overline{a}),\pm w(\overline{b}),\pm w(\overline{c})$. Then, the corners $T$ and $T'$ in such restriction can be easily represented as the first diagram of Figure \ref{fig:patternsavoid}. Clearly, such configuration tell us that $T$ is neither unessential nor minimal. The same idea can be used to prove that the remaining eleven patterns in Figure \ref{fig:patternsavoid} should be avoided.

\begin{figure}[ht]
	\centering
	\includegraphics[scale=0.8]{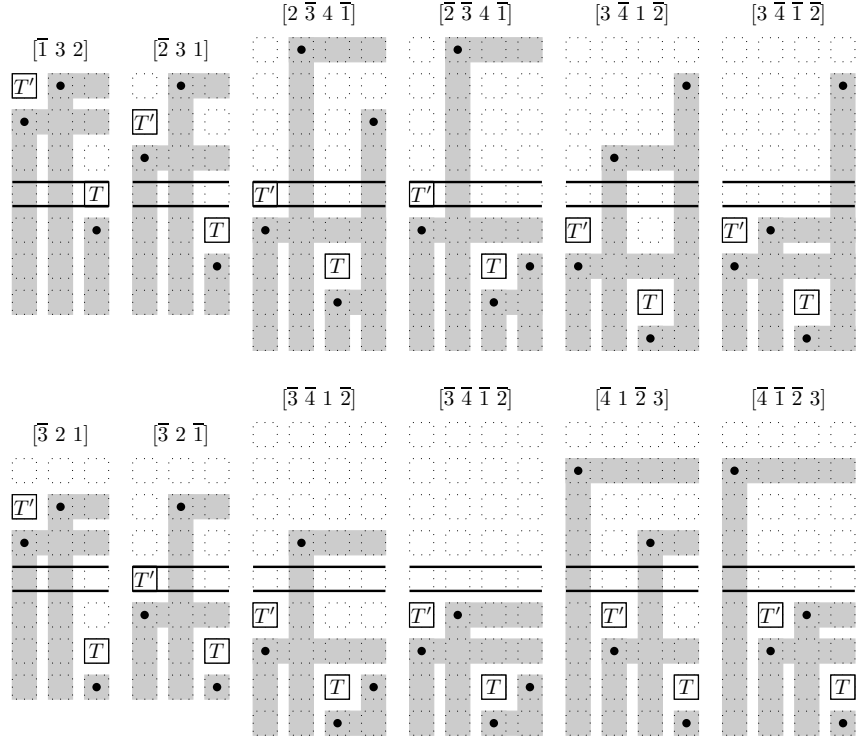}
	\caption{Diagram of $w$ restricted to 11 different patterns.}
	\label{fig:patternsavoid}
\end{figure}

Now, let us assume that $w$ is permutation that avoids all the thirteen patterns listed above. We are going to prove that $\cnr(w)=\NE(w)\cup\uness(w)$, and hence, $w$ is a theta-vexillary permutation.

Suppose that there are corners $T=(p,q)$ and $T'=(p',q')$ such that $q>0>q'$ and $p'>p>0$, i.e., $T$ is in quadrant \textbf{A}, $T'$ is in quadrant \textbf{B}, and $T'<T$.
If we denote $a:={p}$, $b:={p'}-1$ and $c:=w^{-1}(\overline{q'})$, then they satisfy $0<a<b<c$ and $w(a)<0<w(c)<w(b)$. Observe that $\overline{a},\overline{b},\overline{c}$ are the columns of the dots in Figure \ref{pic:3to2_1}.

\begin{figure}[ht]
	\centering
		\includegraphics[scale=0.8]{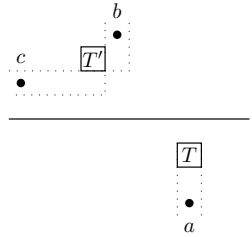}
	\caption{Sketch for the case where $T$ is in quadrant \textbf{A}, $T'$ is in quadrant \textbf{B}, and $T'<T$.}
	\label{pic:3to2_1}
\end{figure}

In order to relate the subsequence $(w(a)\ w(b)\ w(c))$ of $w$ to some 3-pattern $\pi$, we need describe all possible orderings of $\overline{w(a)},w(b)$ and $w(c)$.
\begin{itemize}
\item If ${\overline{w(a)}}<w(c)<w(b)$ then $\pi=[\overline{1}\ 3\ 2]$;
\item If $w(c)<{\overline{w(a)}}<w(b)$ then $\pi=[\overline{2}\ 3\ 1]$;
\item If $w(c)<w(b)<{\overline{w(a)}}$ then $\pi=[\overline{3}\ 2\ 1]$.
\end{itemize}

By hypothesis, the pattern in each case should be avoid. Hence the configuration in Figure \ref{pic:3to2_1} is impossible.

Now, suppose that there are corners $T=(p,q)$ and $T'=(p',q')$ such that $q>q'>0$ and $p'>p>0$, i.e., both $T$ and $T'$ are in quadrant \textbf{A} and $T'<T$. Denote $i:=w^{-1}(\overline{q}+1)$, $a={p}$, $b={p'}-1$ and $c=w^{-1}(\overline{q'})$.

If $i>0$, then they satisfy $0<i<a<b<c$ and $w(a)<w(i)<w(c)<w(b)$. Observe that $\overline{\imath},\overline{a},\overline{b},\overline{c}$ are the columns of the dots in Figure \ref{pic:3to2_2} (left).

\begin{figure}[ht]
	\centering
    \includegraphics[scale=0.8]{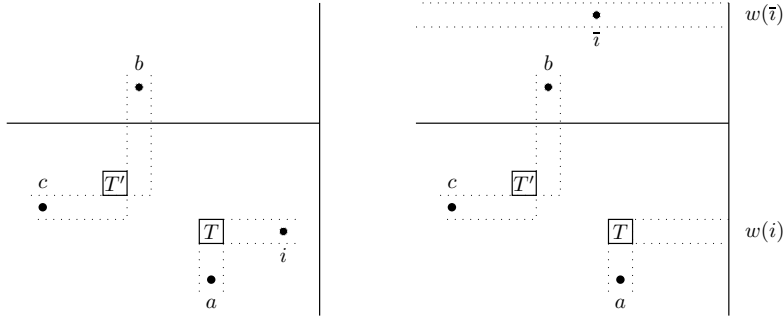}
	\caption{Sketch for the case where both $T$ and $T'$ are in quadrant \textbf{A}, and $T'<T$.}
	\label{pic:3to2_2}
\end{figure}

In order to relate the subsequence $(w(i)\ w(a)\ w(b)\ w(c))$ of $w$ to some 4-pattern $\pi$, we need to describe all possible orderings of $\overline{w(i)}, \overline{w(a)}, \overline{w(c)}$ and $\pm w(b)$.
\begin{itemize}
\item If ${\overline{w(b)}}<\overline{w(c)}<\overline{w(i)}<\overline{w(a)}$ then $\pi=[\overline{3}\ \overline{4}\ \overline{1}\ \overline{2}]$;

\item If ${w(b)}<\overline{w(c)}<\overline{w(i)}<\overline{w(a)}$ then $\pi=[\overline{3}\ \overline{4}\ 1\ \overline{2}]$;

\item If $\overline{w(c)}<{w(b)}<\overline{w(i)}<\overline{w(a)}$ then $\pi=[\overline{3}\ \mathbf{\overline{4}\ 2\ \overline{1}}]$;

\item If $\overline{w(c)}<\overline{w(i)}<{w(b)}<\overline{w(a)}$ then $\pi=[\overline{2}\ \mathbf{\overline{4}\ 3\ \overline{1}}]$;

\item If $\overline{w(c)}<\overline{w(i)}<\overline{w(a)}<{w(b)}$ then $\pi=[\overline{2}\ \overline{3}\ 4\ \overline{1}]$.
\end{itemize}

By hypothesis, the pattern in each case should be avoided (in some cases, the highlighted parts are avoided by $[\overline{3}\ 2\ \overline{1}]$). Hence this configuration is impossible.

If $i<0$ then we have four possibilities to place $\overline{\imath}>0$ in the sequence $0<a<b<c$. Observe that $i,\overline{a},\overline{b},\overline{c}$ are the columns of the dots in Figure \ref{pic:3to2_2} (right). Table \ref{tab:possiblearreng} combines all these possibilities along with all possible orderings of $w(\overline{\imath}), \overline{w(a)}, \overline{w(c)}$ and $\pm w(b)$, in order to get the respective 4-pattern $\pi$ relative to the correspondent subsequence of $w$.

\begin{table}[ht]
\caption{Combinations to get the respective 4-pattern of the subsequence of $w$.}
\label{tab:possiblearreng}
\renewcommand{\arraystretch}{1.3}
\begin{center}
\footnotesize{
\begin{tabular}{|c|c|c|c|c|}
\hline 
 & $\overline{\imath}\sma a\sma b \sma c$ & $a\sma \overline{\imath}\sma b\sma c$ & $a\sma b\sma \overline{\imath}\sma c$ & $a\sma b\sma c\sma {\overline{\imath}}$ \\ 
\hline 
${\overline{w(b)}}\sma \overline{w(c)}\sma {w(\overline{\imath})}\sma \overline{w(a)}$ & $[3\ \overline{4}\ \overline{1}\ \overline{2}]$ &
$[\mathbf{\overline{4}\ 3\ \overline{1}}\ \overline{2}]$ & $[\mathbf{\overline{4}}\ \overline{1}\ \mathbf{3\ \overline{2}}]$ & 
$[\overline{4}\ \overline{1}\ \overline{2}\ 3]$ \\ 
\hline 
${{w(b)}}\sma \overline{w(c)}\sma {w(\overline{\imath})}\sma \overline{w(a)}$ & $[3\ \overline{4}\ 1\ \overline{2}]$ & 
$[\mathbf{\overline{4}\ 3}\ 1\ \mathbf{\overline{2}}]$ & $[\mathbf{\overline{4}}\ 1\ \mathbf{3\ \overline{2}}]$ & 
$[\overline{4}\ 1\ \overline{2}\ 3]$ \\ 
\hline 
$\overline{w(c)}\sma{{w(b)}}\sma  {w(\overline{\imath})}\sma \overline{w(a)}$ & $[3\ \mathbf{\overline{4}\ 2\ \overline{1}}]$ & 
$[\mathbf{\overline{4}\ 3}\ 2\ \mathbf{\overline{1}}]$ & $[\mathbf{\overline{4}\ 2}\ 3\ \mathbf{\overline{1}}]$ & 
$[\mathbf{\overline{4}\ 2\ \overline{1}}\ 3]$ \\ 
\hline 
$\overline{w(c)}\sma {w(\overline{\imath})}\sma{{w(b)}}\sma  \overline{w(a)}$ & $[2\ \mathbf{\overline{4}\ 3\ \overline{1}}]$ & 
$[\mathbf{\overline{4}\ 2}\ 3\ \mathbf{\overline{1}}]$ & $[\mathbf{\overline{4}\ 3}\ 2\ \mathbf{\overline{1}}]$  & 
$[\mathbf{\overline{4}\ 3\ \overline{1}}\ 2]$ \\ 
\hline 
$\overline{w(c)}\sma {w(\overline{\imath})}\sma \overline{w(a)}\sma{{w(b)}}$ & $[2\ \overline{3}\ 4\ \overline{1}]$ & 
$[\mathbf{\overline{3}\ 2}\ 4\ \mathbf{\overline{1}}]$ & $[3\ \mathbf{\overline{4}\ 2\ \overline{1}}]$ & 
$[\mathbf{\overline{3}\ 4}\ \overline{1}\ \mathbf{2}]$ \\ 
\hline
\end{tabular}}
\end{center}
\end{table}

By hypothesis, the pattern in each case should be avoid and, hence, this case is impossible.

Finally, suppose that there are corners $T=(p,q)$ and $T'=(p',q')$ such that $0>q>q'$ and $p'>p>0$, i.e., both $T$ and $T'$ are in quadrant \textbf{B}, and $T'<T$.
If we denote $i=w^{-1}(\overline{q}+1)$, $a={p}-1$, $b={p}$, $c={p'}-1$ and $d=w^{-1}(\overline{q'})$, then they satisfy $i<a<b<c<d$, $w(b)<w(a)$ and $w(b)<w(i)<w(d)<w(c)$. Observe that $\overline{\imath},\overline{a},\overline{b},\overline{c}$ are the columns of the dots in Figure \ref{pic:3to2_3} (left).

\begin{figure}[ht]
	\centering
		\includegraphics[scale=0.8]{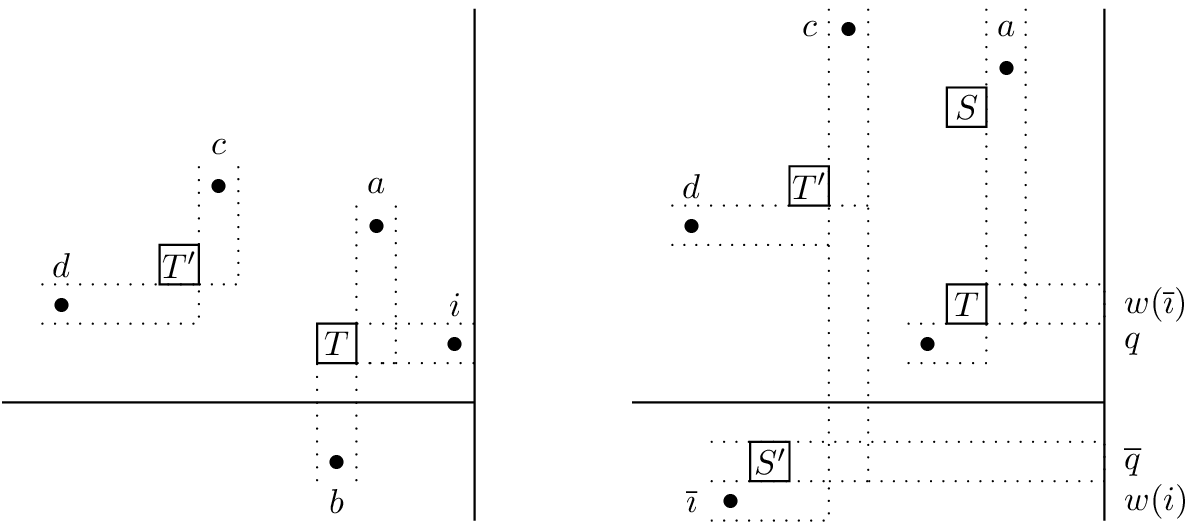}
	\caption{Sketch for the case where both $T$ and $T'$ are in quadrant \textbf{B}, and $T'<T$.}
	\label{pic:3to2_3}
\end{figure}

Consider the following situations:
\begin{itemize}
\item If ${w(b)}<0$ then the subsequence $(w(b)\ w(c)\ w(d))$ of $w$ is a 3-pattern $\pi$ equal to $[\overline{1}\ 3\ 2]$, $[\overline{2}\ 3\ 1]$ or $[\overline{3}\ 2\ 1]$, which is impossible;

\item If $0<w(b)<{w(a)}<w(d)<w(c)$ then  $(w(a)\ w(b)\ w(c)\ w(d))$ is a 4-pattern $\pi=[2\ 1\ 4\ 3]$ and also should be avoided;

\item If $i>0$ and $w(b)>0$ then the subsequence $(w(i)\ w(b)\ w(c)\ w(d))$ is a 4-pattern $\pi=[2\ 1\ 4\ 3]$ and also should be avoided;

\item If $0>i>\overline{c}$ and ${w(b)}>0$ then the subsequence $(w(\overline{\imath})\ w(c)\ w(d))$ of $w$ is a 3-pattern $\pi$ equal to $[\overline{1}\ 3\ 2]$, $[\overline{2}\ 3\ 1]$ or $[\overline{3}\ 2\ 1]$, which is impossible;

\item If $i<\overline{c}$ and $0<w(b)<w(d)<{w(a)}$ then clearly there are SE corners $S$ and $S'$ as in Figure \ref{pic:3to2_3} (right). 
Therefore, such construction implies that $T$ is an unessential box. \qedhere
\end{itemize}
\end{proof}

Therefore, Propositions \ref{prop:svexcore2} and \ref{prop:patavoid} prove Theorem \ref{thm:main}.

This pattern avoidance criterion allow us to compare a theta-vexillary permutation with a vexillary permutation of type B defined by Billey and Lam \cite{BL}. According to them, a signed permutation is vexillary of type B if the Stanley symmetric function $F_{w}$ is equals to the Schur Q-function $Q_{\lambda}$ for some shape $\lambda \vdash \ell(w)$ with distinct parts. In \cite[Theorem 7]{BL}, they proved that a vexillary permutation of type B should avoid eighteen patterns, which includes all thirteen patterns given in our main theorem. Then, every vexillary permutation of type B is theta-vexillary.

Therefore, for an arbitrary theta-vexillary signed permutation $w$, we cannot say that the Stanley symmetric function $F_{w}$ is equals to the Schur Q-function of $w$. In principle, this property is not expected since the equality holds for every vexillary signed permutation as proved by \cite[Corollary 5.2]{AF14}.

\bibliographystyle{abbrv}
\bibliography{biblio}

\begin{thebibliography}{10}

\bibitem{AF12}
D.~Anderson and W.~Fulton.
\newblock {Degeneracy loci, Pfaffians, and vexillary signed permutations in
  types B, C, and D}.
\newblock {\em ArXiv e-prints}, 2012.
\newblock \texttt{arXiv:1210.2066}.

\bibitem{AF15}
D.~Anderson and W.~Fulton.
\newblock Chern class formulas for classical-type degeneracy loci.
\newblock {\em ArXiv e-prints}, 2015.
\newblock \texttt{arXiv:1504.03615}.

\bibitem{AF16}
D.~Anderson and W.~Fulton.
\newblock Diagrams and essential sets for signed permutations.
\newblock {\em ArXiv e-prints}, 2016.
\newblock \texttt{arXiv:1612.08670}.

\bibitem{AF14}
D.~Anderson and W.~Fulton.
\newblock Vexillary signed permutations revisited.
\newblock {\em ArXiv e-prints}, 2018.
\newblock \texttt{arXiv:1806.01230}.

\bibitem{BL}
S.~Billey and T.~K. Lam.
\newblock {Vexillary elements in the hyperoctahedral group}.
\newblock {\em J. Algebraic Combin.}, 8:139--152, 1998.

\bibitem{BB}
A.~Björner and F.~Brenti.
\newblock {\em Combinatorics of Coxeter groups}.
\newblock Springer-Verlag, 2005.

\bibitem{BKT}
A.~S. Buch, A.~Kresch, and H.~Tamvakis.
\newblock {A Giambelli formula for isotropic Grassmannians}.
\newblock {\em Selecta Mathematica}, 23:869--914, 2017.

\bibitem{Fu91}
W.~Fulton.
\newblock {Flags, Schubert polynomials, degeneracy loci, and determinantal
  formulas}.
\newblock {\em Duke Math. J.}, 65:381--420, 1991.

\bibitem{Lam18}
J.~Lambert.
\newblock {\em Combinatorics on Schubert varieties}.
\newblock PhD thesis, Universidade Estadual de Campinas, 2017.

\bibitem{LS}
A.~Lascoux and M.-P. Schützenberger.
\newblock {Polynômes de Schubert}.
\newblock {\em C.R. Acad. Sci. Paris Sér. I Math}, 294:447--450, 1982.

\bibitem{patternDB}
B.~Tenner.
\newblock Database of permutation pattern avoidance.
\newblock URL: http://math.depaul.edu /bridget/patterns.html.

\end{thebibliography}

\end{document}